\documentclass[a4paper,11pt,reqno]{amsart}
\usepackage{amsmath,multicol, a4wide }
\usepackage[all]{xy}

\newtheorem{theorem}{\bf Theorem}[subsection]
\newtheorem{theoremsection}{\bf Theorem}[section]
\newtheorem{prop}[theorem]{\bf Proposition}

\newtheorem{definition}[theorem]{\bf Definition}
\newtheorem{definitionsection}[theoremsection]{\bf Definition}

\theoremstyle{remark}
\newtheorem{example}[theorem]{\bf Example}

\theoremstyle{remark}
\newtheorem{rem}[theorem]{\bf Remark}
\newtheorem{remsection}[theoremsection]{\bf Remark}

 \numberwithin{equation}{subsection}

\newcommand{\ie}{{\it i.e.\;}}

\newcommand{\ad}{\operatorname{ad}}

\def\go{\mathfrak}
\def\bb{\mathbb}
\def\C{\bb C}

\def\N{\bb N}

\def\go{\mathfrak}
\def\bb{\mathbb}

 \def\adots{\mathinner{\mkern2mu\raise1pt\hbox{.}
\mkern3mu\raise4pt\hbox{.}\mkern1mu\raise7pt\hbox{.}}}

 \title[Multiplicity free spaces]{Multiplicity free spaces with a one dimensional quotient}

\author{Hubert Rubenthaler}

\address
{Hubert Rubenthaler\\ Institut de Recherche Math\'ematique Avanc\'ee\\
Universit\'e de Strasbourg et CNRS\\
7 rue Ren\'e Descartes\\
67084 Strasbourg Cedex\\ France\\
E-mail: {\tt rubenth@math.unistra.fr}}

\begin{document} 
\parindent=0pt
 \maketitle


\begin{abstract} The multiplicity free spaces with a one dimensional quotient were introduced by Thierry Levasseur in \cite{Levasseur}.   Recently, the author has shown that the algebra of differential operators on such spaces which are invariant under the semi-simple part of the group is a Smith algebra (\cite{Rubenthaler-mult-free}). We give here the classification of these spaces which are indecomposable, up to geometric equivalence. We also investigate  whether or not these spaces  are regular or of parabolic type as a prehomogeneous vector space.

 \end{abstract}
  {\small \def\contentsname{Table of Contents}
\tableofcontents}
\section{Introduction}

A multiplicity free space is a  representation of a connected reductive group $G$ on a finite dimensional vector space $V$ (everything is defined over $\C$) such that every irreducible representation of $G$ appears at most once in the associated representation of $G $ on the space $\C[V]$ of polynomials on $V$ (see Section 2 for details). For a survey on multiplicity free spaces we refer to \cite{Benson-Ratcliff-survey}. 
Multiplicity free spaces, which were introduced by V. Kac in \cite{Kac}, play now an important role in invariant theory and harmonic analysis (see for example \cite{Howe}, \cite{Howe-Umeda}, \cite{Knop}, the references in \cite{Benson-Ratcliff-survey}, see also \cite{Kobayashi} for a more general concept). Various characterizations of multiplicity free spaces, which are summarized in Theorem \ref{th-caracterisation-MF} below, were obtained by Vinberg-Kimelfeld (\cite{Vinberg-Kimelfeld}), Howe-Umeda (\cite{Howe-Umeda}), Knop (\cite{Knop}).   A corollary of these characterizations is that a multiplicity free space is always a prehomogeneous space (in fact even under a Borel subgroup). This is the reason why prehomogeneous vector spaces occur so often in this paper. The classification of multiplicity free spaces was achieved independently by Benson-Ratcliff (\cite{Benson-Ratcliff-article}) and Leahy (\cite{Leahy}) after partial classifications by Brion (\cite{Brion}) and Kac (\cite{Kac}). 
\vskip 5pt

In this paper we are interested in a specific family of multiplicity free spaces, the so-called multiplicity free spaces with a one dimensional quotient which were introduced by T. Levasseur in \cite{Levasseur}. This means roughly speaking that the categorical quotient $V/\kern -1mm{/}G'$ has dimension one, where $G'$ is the semi-simple part of $G$ (see Definition \ref{def-MF} below). In his paper T. Levasseur proves that if $(G,V)$ is a multiplicity free space with a one dimensional quotient, then the radial component of the (non-commutative) algebra $D(V)^{G'}$ of $G'$-invariant differential operators is a Smith algebra over $\C$. In a recent paper (\cite {Rubenthaler-mult-free}) we showed that  the full algebra $D(V)^{G'}$ is a Smith algebra over its center, which is a polynomial algebra. 

\vskip 5pt

The purpose of this paper is to give the complete classification of all multiplicity free spaces with a one dimensional quotient, including irreducible and non irreducible representations.  Our classification is obtained up to geometric equivalence, which is the natural equivalence relation among multiplicity free spaces. It is worthwhile noticing that the list of irreducibles already appears  in \cite{Levasseur}. Moreover our investigations lead in all cases (irreducibles and non irreducibles) to some extra informations like parabolicity, regularity, and explicit fundamental relative invariants of the underlying prehomogeneous vector spaces.

\vskip 5pt

In section 2 we recall general facts about multiplicity free space. We give first a brief account of the theory of prehomogeneous vector spaces (2.1) and also recall the definition of parabolic type prehomogeneous spaces (2.2), including their weighted Dynkin diagrams which encode many informations (see Definition \ref{def-diagramme} and Remark \ref{rem-diagram}). General  definitions and results about multiplicity free spaces can be found   in 2.3. and 2.4. In 2.5, as an example,  we describe an important family of irreducible multiplicity free space with a one dimensional quotient, namely the irreducible regular prehomogeneous vector spaces of commutative parabolic type.

\vskip 5pt

Section 3 contains the main result, the classification theorem of the multiplicity free spaces with a one dimensional quotient (Theorem \ref{th-classification}). The corresponding lists (Tables 2 and 3) take place at the end of the paper.

\vskip 5pt

Section 4 is devoted to the proof of Theorem \ref{th-classification}. The proof uses case by case examinations from the list by Benson and Ratcliff (\cite{Benson-Ratcliff-article}) and some tools from the theory of prehomogeneous vector spaces.

\vskip 20pt
{\bf Notations: }
In this paper we will denote by $GL(n)$, $SL(n)$, $SO(n)$ the general linear group , the special linear group, the special orhogonal group of complex matrices of size $n$ respectively. As usual we will denote by $Sp(n)$ the symplectic group of $2n\times 2n$ complex matrices. We will also denote by $\go{gl}(n)$, $\go{sl}(n)$, $\go{o}(n)$, $\go{sp}(n)$ the corresponding Lie algebras.
The vector space of $m\times n$ complex matrices will be denoted by $M_{m,n}$, and the space of square $n\times n$ matrices will be denoted by $M_{n}$. Finally $Sym(n)$ will denote the $n\times n$ symmetric matrices and  $AS(n)$ will denote the skew symmetric matrices. If $n$ is even and if $x\in AS(n)$, then $Pf(x) $ stands for the pfaffian of the matrix $x$.


\section{Multiplicity free spaces. Basic definitions and properties}

\vskip 5pt
\subsection {Prehomogeneous Vector Spaces} \hfill
 \vskip5pt

  Let $G$ be a  connected algebraic group over ${\bb C}$, and let $(G,\rho, V)$ be a   rational representation of $G$ on the (finite dimensional) vector space $V$. Then the triplet $(G,\rho,V)$ is called a {\it prehomogeneous vector space} (abbreviated to $PV$) if the action of $G$ on $V$ has a Zariski open orbit $\Omega\in V$.   For the general theory of $PV$'s, we refer the reader to the book  of Kimura \cite{Kimura-book} or to \cite{Sato-Kimura}. The elements in $\Omega$ are called {\it generic}. The $PV$ is said to be {\it irreducible } if  the corresponding representation is irreducible. The {\it singular set} $S$ of  $(G,\rho,V)$ is defined by $S=V\setminus \Omega$. Elements in $S$ are called {\it singular}. If no confusion can arise we often simply denote the $PV$ by $(G,V)$. We will also   write $g.x$ instead of $\rho(g)x$, for $g\in G$ and $x\in V$. It is easy to see that the condition for a rational representation $(G,\rho,V)$ to be a $PV$ is in fact an infinitesimal condition. More precisely let ${\go g}$ be the Lie algebra of $G$ and let $d\rho$ be the derived representation of $\rho$. Then $(G,\rho,V)$ is a PV if and only if there exists $v\in V$ such that the map:
$$\begin{array}{rcl}
{\go g}&\longrightarrow&V\\
X&\longmapsto&d\rho(X)v
\end{array}$$
is surjective (we will often write $X.v$ instead of $d\rho(X)v$). Therefore we will call $({\go g}, V)$ a $PV$ if the preceding condition is satisfied.

Let $(G,V)$ be  a $PV$. A rational function $f$ on $V$ is called a {\it relative invariant} of    $(G,V)$ if there exists a rational character $\chi $ of $G$ such that $f(g.x)=\chi(g)f(x)$ for $g\in G$ and $x\in V$.  From the existence of an open orbit it is easy to see that a character $\chi$ which is trivial on the isotropy subgroup of an element  $x\in \Omega$ determines a unique    relative invariant $P_{\chi}$. Let $S_{1},S_{2},\dots,S_{k}$ denote the irreducible components of codimension one of the singular set $S$. Then there exist irreducible polynomials $P_{1}, P_{2},\dots,P_{k}$ such that $S_{i}=\{x\in V\,|\, P_{i}(x)=0\}$. The $P_{i}$'s are unique up to nonzero constants. It can be proved that the $P_{i}$'s are relative invariants of $(G,V)$ and any nonzero relative invariant $f$ can be written in a unique way $f=cP_{1}^{n_{1}}P_{2}^{n_{2}}\dots P_{k}^{n_{k}}$, where $n_{i}\in {\bb Z}$ and $c\in \C^*$ . The polynomials $P_{1}, P_{2},\dots,P_{k}$ are called the {\it fundamental relative invariants} of $(G,V)$. Moreover if the representation $(G,V)$ is irreducible then  there exists at most one irreducible polynomial (up to multiplication by a non zero constant) which is relatively invariant.

 The prehomogeneous vector space  $(G,V)$ is called {\it regular} if there exists a relative invariant polynomial $P$ whose Hessian $H_{P}(x)$ is nonzero on $\Omega$. If $G$ is reductive, then $(G,V)$ is regular if and only if the singular set $S$ is a hypersurface, or if and only if the isotropy subgroup of a generic point  is reductive. If the $PV$ $(G,V)$ is regular, then the contragredient representation $(G,V^*)$ is again a $PV$. Regular $PV$'s are of particular interest, due to the zeta functions that one can associate to their real forms (\cite{Sato-Shintani}).
 
 \begin{rem}\label{rem-caracteres} Let us mention a well known Lemma from the Theory of $PV$'s , which will be used in section 4. If $(G,V)$ is a $PV$, and if $X_{0}$ is a generic point, then the characters arising as characters of relative invariants are the characters of the quotient group  $G/H$ where $H$ is the normal subgroup of $G$ generated by the derived group $[G,G]$ and the generic isotropy subgroup $G_{X_{0}}$. This group does not depend on $X_{0}.$ For details, see \cite{Kimura-book}, Proposition 2.12. p.28.
 \end{rem}
 
 \vskip 15pt
 
 \subsection{PV's of parabolic type}\hfill
 
 \vskip 5pt
 
 A $PV$ $(G,V)$ is called reductive if the group $G$ is reductive. Among the reductive $PV$'s there is a family of particular interest, the so-called $PV$'s of parabolic type. Let ${\go g}$ be a simple Lie algebra over $\C$. Let ${\go h}$ be a Cartan subalgebra of ${\go g}$, and let $\Sigma$ be the root system of $({\go g},{\go h})$. We fix once
and for all a system of simple roots
$    \Psi $ for $\Sigma$. We denote by $\Sigma^+$ 
(resp. $\Sigma^-$) the corresponding set of positive (resp. negative)
roots in $\Sigma  $. Let
$\theta$ be a
subset of
$\Psi$ and let us   make the standard construction of the   parabolic subalgebra ${\go 
p}_\theta
\subset
{\go g}$ associated to $\theta$. As usual we   denote by $\langle
\theta\rangle$ the set of all roots which are linear combinations of elements in
$\theta$, and put $\langle \theta\rangle^\pm=\langle \theta\rangle\cap \Sigma^\pm$.

Set
\begin{align*} {\go  h}_\theta=\theta^\bot=\{X\in {\go  h}\,|\,\alpha(X)=0
\;\; \forall \alpha\in \theta\},    
 &\quad  {\go  l_\theta}={\go  z}_{{\go g}}({\go  h}_\theta)= {\go  h}\oplus
\sum_{\alpha\in \langle \theta\rangle} {\go g}^\alpha, &{\go 
n}_\theta^\pm=\sum_{\alpha\in \Sigma^\pm\setminus
\langle\theta\rangle^\pm} {\go g}^\alpha
\end{align*}

Then ${\go  p}_\theta= {\go  l}_\theta \oplus {\go  n}_\theta^+$ is  the standard
parabolic subalgebra  associated to $\theta$. There is also a standard ${\bb Z}$-grading
of ${\go g}$ related to these data. Define $H_\theta$ to be the unique element
of ${\go  h}_\theta$ satisfying the linear equations
$$\alpha(H_\theta)=0 \quad \forall \alpha\in \theta\quad {\rm and }  \quad
\alpha(H_\theta)=2 \quad \forall \alpha\in \Psi\setminus \theta.\quad    $$
The above mentioned grading is just the grading obtained from the eigenspace
decomposition of $\ad H_\theta$:
$$d_p(\theta)=\{X\in {\go g} \,|\,[H_\theta,X]=2pX\}.$$
Then we obtain easily:
$${\go g}=\oplus_{p\in{\bb Z}}d_p(\theta),\quad{\go 
l}_\theta=d_0(\theta),\quad{\go  n}_\theta^+=\sum_{p\geq 1} d_p(\theta),\quad{\go 
n}_\theta^-=\sum_{p\leq -1} d_p(\theta).$$
It  is known (using a result of Vinberg \cite{Vinberg}) that $({\go  l}_\theta,d_1(\theta))$
is a prehomogeneous vector space. In fact all the spaces $({\go 
l}_\theta,\,d_p(\theta))$ with ${p\not =0}$ are prehomogeneous, but there is no loss of
generality if we only consider $({\go  l}_\theta,\,d_1(\theta))$. These  spaces have
been called prehomogeneous vector spaces of parabolic type (\cite{Rubenthaler-note-PV}). There are in general
neither irreducible nor regular. But they are of particular interest, because in the
parabolic context, the group (or more precisely its Lie algebra ${\go  l}_\theta$) and
the space (here $d_1(\theta))$ of the $PV$ are embedded into  a rich structure, namely the
simple Lie algebra
${\go g}$. For example the derived representation of the $PV$ is just the   adjoint
representation  of
${\go  l}_\theta$ on $d_1(\theta)$. Moreover the Lie algebra ${\go g}$ also
contains the dual $PV$, namely $({\go  l}_\theta,d_{-1}(\theta))$. 

There is an easy criterion to decide wether or not an irreducible $PV$ of parabolic type is regular and in fact  most of the reduced irreducible  reductive regular $PV$'s from Sato-Kimura  list are of parabolic type (for details we refer to \cite{rub-kyoto},\cite{Rubenthaler-these-etat} and \cite{Rubenthaler-bouquin-PV}).


As these $PV$'s are in one to one correspondence with the subsets $\theta \subset \Psi$, we
will describe them by the mean of the following weighted Dynkin diagram:

\begin{definition} \label{def-diagramme}  
 The diagram of the $PV$ $({\go  l}_\theta,d_1(\theta))$
is the Dynkin diagram of $({\go g},{\go  h})$ $($or $\Sigma$ $)$ where the
vertices corresponding to the simple roots of $\Psi\setminus \theta$ are circled $($see an example
below$)$.
\end{definition}

 This very simple classification  by means of diagrams contains nevertheless some
immediate and  interesting information concerning the $PV$ $({\go  l}_\theta,d_1(\theta))$ (for all these
facts, see \cite{Rubenthaler-note-PV}, \cite{rub-kyoto}, \cite{Rubenthaler-these-etat} or \cite {Rubenthaler-bouquin-PV}):

\begin{rem}\label{rem-diagram}\hfill

a) The Dynkin diagram of $ {\go  l}_\theta'=[{\go 
l}_\theta,{\go  l}_\theta]   $ (\ie the semi-simple part of the Lie algebra of the group)
is the Dynkin diagram of ${\go g}$ where we have removed the circled vertices
and the edges connected to these vertices.

b) In fact as a Lie algebra ${\go  l_\theta ={\go 
l}_\theta}'\oplus{\go  h}_\theta$ and $\dim {\go  h}_\theta=$ the number of circled
vertices.

c)  The number of irreducible components of the representation $({\go l
}_\theta,\,d_1(\theta))$ is also equal to the number of circled roots, and hence the parabolic $PV$ $({\go l}_{\theta}, d_{1}(\theta))$ is irreducible if and only if ${\go p}_{\theta}$ is maximal.  More precisely, if
$\alpha$ is a (simple) circled root, then any nonzero root vector $X_\alpha\in 
{\go g}^\alpha$ generates an irreducible ${\go  l}_\theta$--module $V_\alpha$,
and $d_1(\theta)=\oplus_{\alpha\in\Psi\setminus \theta} V_\alpha$ is the decomposition of
$d_1(\theta)$ into irreducibles.

The decomposition of the representation (${\go  l}_\theta ,d_1(\theta)$) into
irreducibles can also be described by using the eigenspace  decomposition with respect to
$\ad({\go  h}_\theta)$, as we will explain now. For each $\alpha\in {\go  h}^*$, let
$\overline{\alpha}$ be the restriction of $\alpha$ to ${\go  h}_\theta$ and define 
$${\go g}^{\overline{\alpha}}=\{X\in {\go g}\,|\, \forall H \in {\go 
h}_\theta ,\, [H,X]= \overline{\alpha}(H)X\}.$$
Then ${\go g}^{\overline 0}={\go  l}_\theta$ and for $\alpha \in \Psi \setminus
\theta$, we have $V_\alpha = {\go g}^{\overline{\alpha}}$. Hence we can write
$d_1(\theta)=\oplus_{\alpha\in \Psi\setminus \theta}{\go g}^{\overline{\alpha}}.$

 Moreover one can notice (always for $\alpha \in \Psi\setminus \theta$) that
$V_\alpha={\go g}^{\overline{
\alpha}}=\sum_{\beta\in\sigma_1^\alpha}{\go g}^\beta$, where
$\sigma_1^\alpha$ is the set of roots which belong to $\alpha\,+\,\text{span}(\theta)$.

d)  One can  also directly read the highest weight of $V_\alpha$ from
the diagram. The highest weight of $V_\alpha$ relatively to the negative Borel sub-algebra ${\go 
b}_\theta^-={\go  h}\oplus\sum_{\alpha\in\langle \theta\rangle^-}{\go g}^\alpha$ is $\overline{\alpha}=\alpha_{|_{{\go  h}(\theta)}}$. Let $\omega_\beta$
$(\beta\in \theta)$ be the fundamental weights of ${\go  l}_\theta'$ (\ie the dual basis
of $(H_\beta)_{\beta\in \theta}$). For each circled  root $\alpha$ (\ie for
each
$\alpha\in \Psi\setminus \theta $\,), let $J_\alpha=\{(\beta_i)\} $ be the set of roots
in $\theta$ (= non-circled) which are connected to $\alpha$ in the diagram. From
elementary diagram considerations we know that $J_\alpha$ may be empty and that there are
always no more than
$3$ roots in
$J_\alpha$.

          If $J_\alpha=\emptyset$, then $V_\alpha$ is the trivial one dimensional
representation of ${\go  l}_\theta$. 
  
If $J_\alpha\not =\emptyset$, then the highest weight $\overline{\alpha}$ of $V_{\alpha}$ is given by 
$\overline{\alpha}=\sum_{i\in J_\alpha} c_i\omega_{\beta_i}$ where
$c_i=\alpha(H{_{\beta_i}})$ and where $  \alpha( H_{\beta_i})$ can be computed as follows:

$$ (R)
\begin{cases}
\text{ if  } ||\alpha||\leq||\beta_i||  , \text{ then }  \alpha(H_{\beta_i})=-1\ ;\\

\text{ if  }   ||\alpha||>||\beta_i||\text{  and if } \alpha \text{ and  } \beta_i \text{ are
connected by }
 j \text{ arrows }  ( 1\leq j\leq 3), \\
  \quad\text{ then }  \alpha(H_{\beta_i})=-j\ . 
\end{cases}
$$
\end{rem}
Let us illustrate this with an example.

\begin{example} \label{exdiagram}  Consider the following diagram:

 
 
 
 
 

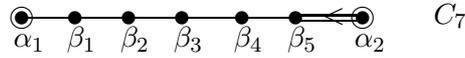
\raisebox{8pt}{
\hbox{\unitlength=0.5pt
\hskip 100pt\begin{picture}(300,30)
\put(4,-12){$\alpha_{1}$}
\put(10,10){\circle*{10}}
\put(10,10){\circle{16}}
\put(15,10){\line (1,0){30}}
\put(44,-12){$\beta_{1}$}
\put(50,10){\circle*{10}}
\put(55,10){\line (1,0){30}}
\put(84,-12){$\beta_{2}$}
\put(90,10){\circle*{10}}
\put(95,10){\line (1,0){30}}
\put(124,-12){$\beta_{3}$}
\put(130,10){\circle*{10}}
\put(135,10){\line(1,0){35}}
\put(169,-12){$\beta_{4}$}
\put(175,10){\circle*{10}}
\put(180,10){\line (1,0){30}}
\put(209,-12){$\beta_{5}$}
\put(215,10){\circle*{10}}
\put(219,12){\line (1,0){41}}
\put(219,8){\line(1,0){41}}
\put(235,5.5){$<$}
\put(265,10){\circle*{10}}
\put(265,10){\circle{16}}
\put(259,-12){$\alpha_{2}$}
\end{picture} \,\,\,\raisebox{3pt} {$C_{7}$}
}}

 \end{example}
 \vskip 5pt
 This diagram   is the diagram of a $PV$ of parabolic type inside ${\go g}\simeq {\go{sp}}(7)\simeq C_{7}$. The Lie algebra ${\go l}_{\theta}$ is isomorphic to $A_{5}\oplus {\go h}_{\theta}\simeq \go{sl}(6)\oplus {\go h}_{\theta}$ where $\dim {\go h}_{\theta}=$ number of circled roots $=2$. There are two irreducible components $V_{\alpha_{1}}$ and $V_{\alpha_{2}}$, and the highest weight of $(A_{5},V_{\alpha_{1}})$ (resp. $(A_{5},V_{\alpha_{2}}$)) relatively to the Borel subalgebra ${\go b}_{\theta}^-$ is  $\omega_{1}$ (resp. $2\omega_{5}$), where  $\omega_{i}$ (i=1,\dots,5) are the fundamental weights of $A_{5}$ corresponding respectively to $\beta_{1},\dots,\beta_{5}$.


\subsection{Multiplicity free spaces}\hfill
\vskip 5pt

 For the results concerning multiplicity free spaces we refer the reader to the survey by Benson and Ratcliff (\cite{Benson-Ratcliff-survey}) or to \cite{Knop} .  
Let $(G,V)$ be a finite dimensional rational  representation of a  connected reductive algebraic group $G$. Let $\C[V]$ be the algebra of polynomials on $V$. Then $G$ acts on $\C [V]$ by 
$$g.\varphi(x)=\varphi(g^{-1}x)\qquad(g\in G, \varphi\in \C[V]).$$
As the space $\C[V]^n$ of homogeneous polynomials of degree $n$ is stable under this action, the representation $(G,\C[V])$ is completely reducible.
Let  $D(V)$ be the algebra of differential operators with polynomial coefficients. The group $G$ acts also on $D(V)$ by 
$$(g.D)(\varphi)= g.(D(g^{-1}.\varphi))\quad (g\in G, D\in D(V),  \varphi\in \C[V]).$$

\begin{definition} Let $G$ be a  connected reductive algebraic group, and let $V$ be the space of a finite dimensional $($complex$)$ rational representation of $G$. The representation $(G,V)$ is said to be multiplicity free if each irreducible representation of $G$ occurs at most once in the representation $(G, \C[V])$.
\end{definition}

From now on "multiplicity-free" will be abbreviated to "$MF$".

Let us give  some  results concerning $MF$  spaces (see \cite{Benson-Ratcliff-survey}, \cite{Howe-Umeda}, \cite{Knop}):
\begin{theorem}\label{th-caracterisation-MF}\hfill

$1)$  A finite dimensional representation $(G,V)$ is $MF$ if and only if $(B,V)$ is a  prehomogeneous vector space for any Borel subgroup $B$ of $G$ $($and hence each $MF$ space $(G,V)$  is a $PV$$)$.

$2)$  A finite dimensional representation $(G,V)$ is $MF$ if and only if the algebra   $D(V)^G$ of invariant differential operators with polynomial coefficients is commutative.

$3)$ If $(G,V)$ is a MF space, then the dual space $(G,V^*)$  is also MF.  More generally if $(G, W\oplus V)$ is a representation of $G$ where $W$ and $V$ are $G$-stable, then  $(G, W\oplus V)$ is MF if and only if  $(G, W\oplus V^*)$ is MF.
\end{theorem}

\begin{proof} (Indications)

Part $1)$ is due to Vinberg and Kimelfeld (\cite{Vinberg-Kimelfeld}), another proof can be found in \cite{Knop} .
Part $2)$  is due to Howe and Umeda (\cite{Howe-Umeda}, Proposition 7.1). The first assertion of Part $3)$, also noted in \cite{Howe-Umeda}, is a consequence of the $G$-equivariant isomorphism $  \C^{i}[V^*]\simeq (\C^{i}[V])^*$. The second assertion of $3)$ is Corollary 3.3 in \cite{Leahy}.

\end{proof}
 Note that the commutativity of $D(V)^G$ for a $MF$ space is just a consequence of the definition, since we have a simultaneaous diagonalization of all the operators in $D(V)^G$.
 \vskip 5pt

If $(G,V)$ is a $MF$ space, and if $B$ is a Borel subgroup of $G$, then, as we have seen $(G,V)$ and $(B,V)$ are prehomogeneous spaces. Let us  denote by $\Delta_{0},\Delta_{1}, \dots,\Delta_{k}, \dots,\Delta_{r}$  the fundamental relative invariants invariants of the $PV $ $(B,V)$, indexed in such a way that $\Delta_{0}, \Delta_{1},\dots,\Delta_{k}$ are the fundamental relative invariants of the $PV$ $(G,V)$ and such that the other invariants are ordered by decreasing degree. It is worthwhile noticing that at least $\Delta_{r}$ is of degree one as the highest weight vectors of the irreducible components of $V^*$ must occur. The nonnegative integer $r+1$ (= the number of fundamental relative invariants under $B$) is called the {\it rank} of the $MF$ space $(G,V)$.
\vskip 15pt

   \subsection{Multiplicity free spaces with a      one dimensional quotient}\hfill

   Let us now  define the main objects this paper deals with.
  \begin{definition}\label{def-MF}
   {\rm (T. Levasseur \cite{Levasseur})}\hfill
   
   $1)$ A prehomogeneous vector space $(G,V)$  is said to be of rank one \footnote{It is worth noticing that if $(G,V)$ is  multiplicity free, then its  rank as a $PV$ is not at all the same as its rank as a $MF$ space.} if there exists an homogeneous polynomial $\Delta_{0}$ on $V$ such that $\Delta_{0}\notin \C[V]^G$ and such that $\C[V]^{G'}=\C[\Delta_{0}]$.

  $2)$ A multiplicity free space $(G,V)$   is said to have a one-dimensional quotient if it is a $PV$ of rank one. (This implies that  the categorical quotient $V/\kern -1mm{/}G'$ has dimension $1$.)
 \end{definition}   
 
 Although the following result is implicit in \cite{Levasseur} we provide a proof here.
 
 \begin{prop}\label{prop-pv-rang-1}\hfill

 If $(G,V)$ is a $PV$ of rank one, then
  the polynomial $\Delta_{0}$ is the unique fundamental relative invariant of $(G,V)$. More precisely a $PV$ $(G,V)$ is of rank one if and only if it has a unique fundamental relative invariant.
 \end{prop}
 
 \begin{proof} We can write $G=G'C$ where $G'=[G,G]$ is the derived group of $G$ and where $C=Z(G)^\circ\simeq (\C^*)^p$ is the connected component of the center of $G$. Let $g\in C$. Then $\Delta_{0}(g^{-1}x)$ is again $G'$-invariant. As $\C[V]^{G'}=\C[\Delta_{0}]$ and as $\Delta_{0}(g^{-1}x)$ has the same degree as $\Delta_{0}$ we obtain that $\Delta_{0}(g^{-1}x)=\chi(g)\Delta_{0}(x)$ with $\chi(g)\in \C^*$. Therefore $\Delta_{0}$ is a relative invariant. Suppose that $\Delta_{0}$ is not irreducible. Then $\Delta_{0}=P_{1}\dots P_{m}$, where the polynomials $P_{i}$ are irreducible relative invariants and $\partial^\circ(P_{i})<\partial^¡(\Delta_{0})$. We should have $P_{i}\in \C[\Delta_{0}]$, which is impossible. Hence $\Delta_{0}$ is irreducible. If $f$ is another fundamental relative invariant then we would have $f\in \C[\Delta_{0}]$ which is impossible. 
 
 It remains to prove that if a $PV$ $(G,V)$ has a unique fundamental relative invariant $\Delta_{0}$ then it is of rank one. As $\Delta_{0}$ is non constant we have of course that $\Delta_{0}\notin \C[V]^G$. Let $P\in \C[V]^{G'}$. If $P=P_{0}+P_{1}+\dots+P_{m}$ whith  $\partial^\circ(P_{i})=i$,  then each $P_{i} $ is $G'$-invariant.  Therefore we can suppose that $P$ has fixed degree $n$ (i.e. $P\in \C[V]^{G'}\cap \C[V]^n$).
 
 Let
 $$\C[V]^n=\bigoplus_{i=0}^p M_{i}$$
 be the decomposition of $\C[V]^n$ into $G'$-isotypic components. We suppose that $M_{0}$ is the isotypic component of the trivial $G'$-module ($M_{0}=\C[V]^{G'}\cap \C[V]^n$). Hence $P\in M_{0}$. As $G=CG'$, the group $G$ stabilizes each $M_{i}$. Therefore we can write
 $$M_{0}=\bigoplus M_{\chi}$$
 where the $G$-isotypic components $M_{\chi}$ of $M_{0}$ are indexed by characters $\chi$ of $G$ and given by 
 $$M_{\chi}=\{\varphi\in M_{0} \,|\, \varphi(z^{-1}x)=\chi(z)\varphi(x), \forall z\in C, x\in V\}.$$
Hence $P=\sum P_{\chi}, \, P_{\chi}\in M_{\chi}$, and for $z\in C, g'\in G'$ and $x\in V$ we have $P_{\chi}(zg'x)=\chi^{-1}(z)P_{\chi}(g'x)=\chi^{-1}(z)P_{\chi}(x)$. Therefore each $P_{\chi} $ is a relative invariant. But   $(G,V)$ has a unique fundamental relative invariant namely $\Delta_{0}$. Hence $P_{\chi}=c_{\chi}\Delta_{0}^j$ ($c_{\chi}\in \C$). The exponent $j$ does not depend on $\chi$, since all the $P_{\chi}$'s have the same degree. Therefore all the characters $\chi $ are the same, namely $\chi= \lambda_{0}^j$ where $\lambda_{0}$ is the character of $\Delta_{0}.$ This implies that $M_{0}=M_{\lambda_{0}^j}$, and that $P=c\Delta_{0}^j$. Hence $\C[V]^{G'}=\C[\Delta_{0}]$.
 
 \end{proof}

The following result gives a criterion to decide whether or not a $PV$ has rank one. It   will be useful in section $4$ for the classification of  the $MF$ spaces with a one dimensional quotient.

\begin{prop}\label{prop-isotropie-rang1}\hfill

Let $G$ be a connected algebraic group and let   $(G,V)$ be a $PV$. We suppose that $V=V_{1}\oplus V_{2}$ where $V_{1}$ and $V_{2}$ are $G$-stable subspaces, and that $(G,V_{1})$ is a rank one $PV$. Let $(x_{0},y_{0})$ be a generic element in $V$, with $x_{0}\in V_{1}$ and $y_{0}\in V_{2}$.  Let $G_{x_{0}}$ be the isotropy subgroup of $x_{0}$.  We suppose also that the prehomogeneous vector space $(G_{x_{0}}, V_{2})$ has no nontrivial relative invariant $($this property does not depend on the choice of $(x_{0},y_{0})$$)$. Then $(G,V)$ is a rank one $PV$.

\end{prop}

\begin{proof} 

For  the fact that $(G_{x_{0}},V_{2})$ is a $PV$ and that $y_{0} $ is generic for this $PV$ we refer to \cite{Rubenthaler-decomposition}.
As $(G,V_{1})$ is a rank one $PV$ it has a unique fundamental relative invariant $f(x)$ by Proposition \ref{prop-pv-rang-1}. Define $\Delta_{0}(x,y)=f(x)\, (x\in V_{1},y \in V_{2})$. Then, as it is irreducible,  $\Delta_{0}(x,y)$ is a fundamental relative invariant of $(G,V)$. Let $\varphi(x,y)$ be a fundamental relative invariant of $(G,V)$ and consider the function $y\longmapsto \varphi(x_{0},y)$. For $g\in G_{x_{0}}$, we have $\varphi(gx_{0},gy)=\varphi(x_{0},gy)=\chi_{_\varphi}(g)\varphi(x_{0},y)$. Hence $y\longmapsto \varphi(x_{0},y)$ is a relative invariant of $(G_{x_{0}},V_{2})$. But  by hypothesis $(G_{x_{0}},V_{2})$ has no nontrivial relative invariant, hence for all $y\in V_{2}$, we have $\varphi(x_{0},y)=\psi(x_{0})$ (constant with respect to $y$). But as this is true for any generic $x_{0} \in V_{1}$, we obtain that $\varphi(x,y)=\psi(x)$ , for all $x\in V_{1}$ and all $y\in V_{2}$. In other words $\varphi$ does only depend on the $x$ variable. As $\varphi$ is irreducible, so is also $\psi$. And $\psi$ is then a relative invariant of $(G,V_{1})$, hence $\psi=cf$ $(c\in \C)$, or equivalently $\varphi(x,y)=c\Delta_{0}(x,y)$. Then Proposition \ref{prop-pv-rang-1} implies that $(G,V)$ is a rank one $PV$.

\end{proof}

{\bf Notation:} If $(G,V)$ is a $MF$ space with a one dimensional quotient, we will sometimes say that $(G,V)$ is $QD1$.
\subsection{An exemple: the regular commutative PV's of  parabolic type}\hfill

\vskip 5pt

Among the $PV$'s of parabolic type there is a family, the so-called  regular commutative $PV$'s of parabolic type, which are MF spaces with a one dimensional quotient. We will   give here a brief description of these objects.
Notations and conventions are the same as in section $2.2$. The $PV$'s of parabolic type we are going to describe are irreducible. Therefore there is only one circled root which we denote by $\alpha_{0}$ (and then $\theta=\Psi\setminus \{\alpha_{0}\}$). In this section we will impose the extra condition that the coefficient of $\alpha_{0}$ in the highest root is $1$. This implies that $d_{p}(\theta)=\{0\}$ for $p\neq 0,1,-1$. Hence ${\go p}_{\theta}= {\go l}_{\theta}\oplus d_{1}(\theta)$, and the nilradical $d_{1}(\theta)$ of ${\go p}_{\theta}$ is a commutative subalgebra. Therefore the spaces $({\go l}_{\theta},d_{1}(\theta))$ are called  commutative $PV$'s of parabolic type. It is known that these $PV$'s are all MF spaces (\cite{Muller-Rubenthaler-Schiffmann}). By Proposition \ref{prop-pv-rang-1}   those which have a one dimensional quotient are exactly those which have a non trivial relative invariant. From \cite{Muller-Rubenthaler-Schiffmann}   these are also exactly those which are regular, and the list is given in Table 1 below. As they are irreducible, and as the center of ${\go l}_{\theta}$ is one-dimensional, these spaces are automatically indecomposable and saturated (see Definition \ref{equiv-MF} below).  
$$\begin{array}{c}
 \text{\bf Table 1}\\
 \text{\bf Regular PV's of commutative parabolic type}\\
 \\
\begin{array}{|c|c|c|c|}
 
\hline
 &{\go g}& {\go l}_{\theta} & d_{1}(\theta)\\
\hline 
 A_{2n+1}&\hbox{\unitlength=0.5pt
\hskip-100pt \begin{picture}(400,30)(0,10)
\put(90,10){\circle*{10}}
\put(85,-10){$\alpha_1$}
\put(95,10){\line (1,0){30}}
\put(130,10){\circle*{10}}
 
\put(140,10){\circle*{1}}
\put(145,10){\circle*{1}}
\put(150,10){\circle*{1}}
\put(155,10){\circle*{1}}
\put(160,10){\circle*{1}}
\put(165,10){\circle*{1}}
\put(170,10){\circle*{1}}
\put(175,10){\circle*{1}}
\put(180,10){\circle*{1}}
 \put(195,10){\circle*{10}}
 
\put(195,10){\line (1,0){30}}
\put(230,10){\circle*{10}}
\put(220,-10){$\alpha_{n+1}$}
\put(230,10){\circle{18}}
\put(235,10){\line (1,0){30}}
\put(270,10){\circle*{10}}
 
\put(280,10){\circle*{1}}
\put(285,10){\circle*{1}}
\put(290,10){\circle*{1}}
\put(295,10){\circle*{1}}
\put(300,10){\circle*{1}}
\put(305,10){\circle*{1}}
\put(310,10){\circle*{1}}
\put(315,10){\circle*{1}}
\put(320,10){\circle*{1}}
\put(330,10){\circle*{10}}
 
\put(335,10){\line (1,0){30}}
\put(370,10){\circle*{10}}
\put(360,-10){$\alpha_{2n+1}$}
\end{picture} \hskip -50pt
} 
&{\go s}{\go l}(n+1)\times{\go g}{\go l}(n+1)&M_{n}\\
&&&\\
\hline
B_{n}&\hbox{\unitlength=0.5pt
\hskip -100pt\begin{picture}(300,30)(-82,0)
\put(10,10){\circle*{10}}
\put(10,10){\circle{18}}
\put(15,10){\line (1,0){30}}
\put(50,10){\circle*{10}}
\put(55,10){\line (1,0){30}}
\put(90,10){\circle*{10}}
\put(95,10){\line (1,0){30}}
\put(130,10){\circle*{10}}
\put(135,10){\circle*{1}}
\put(140,10){\circle*{1}}
\put(145,10){\circle*{1}}
\put(150,10){\circle*{1}}
\put(155,10){\circle*{1}}
\put(160,10){\circle*{1}}
\put(165,10){\circle*{1}}
\put(170,10){\circle*{10}}
\put(174,12){\line (1,0){41}}
\put(174,8){\line (1,0){41}}
\put(190,5.5){$>$}
\put(220,10){\circle*{10}}
\put(240,7){$B_{n}\,(n\geq 2)$}
 \end{picture} 
}
&{\go s}{\go o}(2n-1)\times {\bb C}&{\bb C}^{2n-1}\\
\hline

 C_{n}&\hbox{\unitlength=0.5pt
\hskip -100pt\begin{picture}(300,30)(-82,0)
\put(10,10){\circle*{10}}
\put(15,10){\line (1,0){30}}
\put(50,10){\circle*{10}}
\put(55,10){\line (1,0){30}}
\put(90,10){\circle*{10}}
\put(95,10){\circle*{1}}
\put(100,10){\circle*{1}}
\put(105,10){\circle*{1}}
\put(110,10){\circle*{1}}
\put(115,10){\circle*{1}}
\put(120,10){\circle*{1}}
\put(125,10){\circle*{1}}
\put(130,10){\circle*{1}}
\put(135,10){\circle*{1}}
\put(140,10){\circle*{10}}
\put(145,10){\line (1,0){30}}
\put(180,10){\circle*{10}}
\put(184,12){\line (1,0){41}}
\put(184,8){\line(1,0){41}}
\put(200,5.5){$<$}
\put(230,10){\circle*{10}}
\put(230,10){\circle{18}}
\put(250, 6){$C_{n}$}
 
\end{picture}
}
&{\go g}{\go l}(n)&Sym(n)\\
 \hline
D_{n}^1&\hbox{\unitlength=0.5pt
\hskip -90pt\begin{picture}(340,40)(-82,8)
\put(10,10){\circle*{10}}
\put(10,10){\circle{18}}
\put(15,10){\line (1,0){30}}
\put(50,10){\circle*{10}}
\put(55,10){\line (1,0){30}}
\put(90,10){\circle*{10}}
\put(95,10){\line (1,0){30}}
\put(130,10){\circle*{10}}
\put(140,10){\circle*{1}}
\put(145,10){\circle*{1}}
\put(150,10){\circle*{1}}
\put(155,10){\circle*{1}}
\put(160,10){\circle*{1}}
\put(170,10){\circle*{10}}
\put(175,10){\line (1,0){30}}
\put(210,10){\circle*{10}}
\put(215,14){\line (1,1){20}}
\put(240,36){\circle*{10}}
\put(215,6){\line (1,-1){20}}
\put(240,-16){\circle*{10}}
\put(260,6){$D_{n}\,(n\geq 4)$}
 \end{picture}
}
&{\go s}{\go o}(2n-2)\times {\bb C}&{\bb C}^{2n-2}\\

&&&\\
\hline
 D_{2n}^2  &\hbox{\unitlength=0.5pt
\hskip-65pt\begin{picture}(400,35)(-82,10)
\put(10,0){\circle*{10}}
\put(15,0){\line (1,0){30}}
\put(50,0){\circle*{10}}
\put(55,0){\line (1,0){30}}
\put(90,0){\circle*{10}}
\put(95,0){\line (1,0){30}}
\put(130,0){\circle*{10}}
\put(140,0){\circle*{1}}
\put(145,0){\circle*{1}}
\put(150,0){\circle*{1}}
\put(155,0){\circle*{1}}
\put(160,0){\circle*{1}}
\put(170,0){\circle*{10}}
\put(175,0){\line (1,0){30}}
\put(210,0){\circle*{10}}
\put(215,4){\line (1,1){20}}
\put(240,26){\circle*{10}}
\put(240,26){\circle{18}}
\put(215,-4){\line (1,-1){20}}
\put(240,-26){\circle*{10}}
\put(250,0){$D_{2n}$ ($ n\geq 2$)}
\end{picture}
}
&{\go g}{\go l}(2n)&AS(2n)\\
&&&\\
&&&\\
\hline
 E_7&\hbox{\unitlength=0.5pt
 \begin{picture}(350,35)(-82,-10)
\put(10,0){\circle*{10}}
\put(15,0){\line  (1,0){30}}
\put(50,0){\circle*{10}}
\put(55,0){\line  (1,0){30}}
\put(90,0){\circle*{10}}
\put(90,-5){\line  (0,-1){30}}
\put(90,-40){\circle*{10}}
\put(95,0){\line  (1,0){30}}
\put(130,0){\circle*{10}}
\put(135,0){\line  (1,0){30}}
\put(170,0){\circle*{10}}
\put(175,0){\line (1,0){30}}
\put(210,0){\circle*{10}}
\put(210,0){\circle{18}}
\put(240,-4){$E_7$}
\end{picture}
}
 &E_{6}\times {\bb C}&{\bb C}^{27}\\
&&&\\
&&&\\
\hline
\end{array}
\end{array}$$
\vskip 10pt


\section{The classification}
 
Let us now  explain the classification of $MF$ spaces with a one dimensional quotient. We begin to describe briefly the classification of all  $MF$ spaces.  Kac (\cite{Kac}) determined all the cases where the representation $(G,V)$ is irreducible. Brion (\cite{Brion}) did the case where $G'=[G,G]$ is (almost) simple. Finally Benson-Ratcliff and Leahy did the rest, independently (\cite{Benson-Ratcliff-survey}, \cite{Benson-Ratcliff-article},\cite{Leahy},\cite{Knop}). 

\begin{definitionsection}\label{equiv-MF}$($see \cite{Knop} $)$\hfill

$1)$ Two representations $(G_{1},\rho_{1},V_{1})$ and $(G_{2},\rho_{2},V_{2})$ are called {\it geometrically equivalent} if there is an isomorphism $\Phi:V_{1}\longmapsto V_{2}$ such that $\Phi(\rho_{1}(G_{1}))\Phi^{-1}= \rho_{2}(G_{2})$.

$2)$ A representation $(G,V)$ is called {\it decomposable} if it is geometrically equivalent to a representation of the form $(G_{1}\times G_{2}, V_{1}\oplus V_{2}$, where $V_{1}$ and $V_{2}$ are non-zero. It is called {\it  indecomposable} if it is not decomposable.

$3)$ A representation $(G,V)$ is called {\it saturated} if the dimension of the center of $\rho(G)$ is equal to the number of irreducible summands of $V$.
\end{definitionsection}

\begin{remsection}\label{rem-geom-equ}\hfill

The notion of geometric equivalence is quite  natural, once one has remarked that the notion of $MF$ space depends only on $\rho(G)$.  Is is worthwhile pointing out that  any representation is geometrically equivalent to its dual representation (see Theorem \ref{th-caracterisation-MF}). Finally note that any representation can be made saturated by adding a torus.

\end{remsection}

\begin{theoremsection}\label{th-classification} The  complete list, up to geometric equivalence, of indecomposable saturated $MF$ spaces with a one dimensional quotient is given by Table $2$ (irreducibles) and Table $3$ (non irreducibles) at the end of the paper.
\end{theoremsection}

\section{Proof}

 This section is devoted to the proof of Theorem \ref{th-classification}.  The classification tables show that indecomposable saturated $MF$ spaces are either irreducible (see for example Table I p. 153 in \cite{Benson-Ratcliff-article}) or they have two irreducible summands (see Table II p. 154 in \cite{Benson-Ratcliff-article}). Using Proposition \ref{prop-pv-rang-1}, we have only to decide whether or not a given  $MF$-spaces  has a unique fundamental invariant. 
  
   For the irreducible $MF$ spaces we have checked this by a case by case computation. For the non-irreducible $MF$ spaces, we also check which one have a one dimensional quotient by a case by case examination, using sometimes the criterion given by Proposition \ref{prop-isotropie-rang1}. 
   
   We also indicate whether or not these spaces are regular of of parabolic type as a prehomogeneous vector spaces.

\subsection{Irreducible MF spaces}\hfill

As said above we examine case by case the $MF$  spaces occuring   in Table 1 of the paper by Benson-Ratcliff  (\cite{Benson-Ratcliff-article}). Of course the similar tables of Leahy (\cite{Leahy}) or of Knop (\cite{Knop}) could have been used instead. Except that  we explicitly mention the center $\C^*$,  we adopt the  notations of Benson and Ratcliff. The notations  $SL(n), SO(n),Sp(n)$ will not only stand for the  groups but also  for the natural representation of the corresponding group on $\C^n, \C^n, \C^{2n}$ respectively. The notation $S^2(SL(n))$ denotes the "natural" representation of $SL(n)$ of the space $Sym(n)$ of symmetric matrices of size $n$, whereas 
$\Lambda^2(SL(n))$ stands for the "natural" representation of $SL(n)$ on the space $AS(n)$ of skew-symmetric matrices of size $n$. Also $G_{2}$ stands for the $7$-dimensional representation of $G_{2}$, and $E_{6}$ denotes the $27$-dimensional representation of $E_{6}$. The notation $H^*$  denotes the contragredient representation of the representation $H$ and the notation $H_{1}\otimes H_{2}$ denotes the tensor product of the corresponding representations of $H_{1}$ and $H_{2}$.    

Recall also that we say that a $MF$ space is $QD1$ if it has a one dimensional quotient.
\vskip5pt
\subsubsection{} $ SL(n) \times   \C^*  (n\geq 1)$.

It is well known that, for $n>1$, this representation has no non-trivial relative invariant. For $n=1$, the corresponding representation is $\C^*$ acting on $\C$, which has obviously one fundamental relative invariant, and hence is $QD1$. It is of commutative parabolic type, corresponding to the diagram $\hbox{\unitlength=0.5pt
 \begin{picture}(0,20)(205,-5).
 
\put(210,0){\circle*{10}}
\put(210,0){\circle{18}}
 
\end{picture}
}
$
\,\,. This is the case $m=n=1$ of 4.1.6. below. It is case $A_{1}$ in Table 1, and a particular case of $(4)$ in Table 2.
\vskip5pt
\subsubsection{} $SO(n)\times \C^* (n\geq3)$.	

 It is well known that the natural representation of $SO(n)\times \C^*$ on $\C^n$ has a unique fundamental relative invariant, namely the nondegenerate quadratic form $Q$. Therefore it is $QD1$. It is also well known that this space is a $PV$ of  commutative parabolic type, corresponding to the diagrams   \hbox{\unitlength=0.5pt
\hskip-30pt\begin{picture}(300,30)(-82,6)
\put(10,10){\circle*{10}}
\put(10,10){\circle{18}}
\put(15,10){\line (1,0){30}}
\put(50,10){\circle*{10}}
\put(55,10){\line (1,0){30}}
\put(90,10){\circle*{10}}
\put(95,10){\line (1,0){30}}
\put(130,10){\circle*{10}}
\put(135,10){\circle*{1}}
\put(140,10){\circle*{1}}
\put(145,10){\circle*{1}}
\put(150,10){\circle*{1}}
\put(155,10){\circle*{1}}
\put(160,10){\circle*{1}}
\put(165,10){\circle*{1}}
\put(170,10){\circle*{10}}
\put(174,12){\line (1,0){41}}
\put(174,8){\line (1,0){41}}
\put(190,5.5){$>$}
\put(220,10){\circle*{10}}
\put(240,7){$B_{p+1}$}
 \end{picture}\hskip 35pt} if $n=2p+1$ 
 and
  \hbox{\hskip -40pt\unitlength=0.5pt
\begin{picture}(340,40)(-82,5)
\put(10,10){\circle*{10}}
\put(10,10){\circle{18}}
\put(15,10){\line (1,0){30}}
\put(50,10){\circle*{10}}
\put(55,10){\line (1,0){30}}
\put(90,10){\circle*{10}}
\put(95,10){\line (1,0){30}}
\put(130,10){\circle*{10}}
\put(140,10){\circle*{1}}
\put(145,10){\circle*{1}}
\put(150,10){\circle*{1}}
\put(155,10){\circle*{1}}
\put(160,10){\circle*{1}}
\put(170,10){\circle*{10}}
\put(175,10){\line (1,0){30}}
\put(210,10){\circle*{10}}
\put(215,14){\line (1,1){20}}
\put(240,36){\circle*{10}}
\put(215,6){\line (1,-1){20}}
\put(240,-16){\circle*{10}}
\put(260,6){$D_{p+1} $}
 \end{picture}
\hskip 30pt} if $n=2p.$

\vskip 15pt
This corresponds   to case $(1)$ in Table 2.

\vskip5pt
\subsubsection{} $Sp(n)\times \C^* (n\geq 2)$.	

This $PV$ is of parabolic type (corresponding to the diagram \hbox{\unitlength=0.5pt
\begin{picture}(240,30)(-6,5)
\put(10,10){\circle*{10}}
\put(10,10){\circle{16}}
\put(15,10){\line (1,0){30}}
\put(50,10){\circle*{10}}
\put(55,10){\line (1,0){30}}
\put(90,10){\circle*{10}}
\put(95,10){\circle*{1}}
\put(100,10){\circle*{1}}
\put(105,10){\circle*{1}}
\put(110,10){\circle*{1}}
\put(115,10){\circle*{1}}
\put(120,10){\circle*{1}}
\put(125,10){\circle*{1}}
\put(130,10){\circle*{1}}
\put(135,10){\circle*{1}}
\put(140,10){\circle*{10}}
\put(145,10){\line (1,0){30}}
\put(180,10){\circle*{10}}
\put(184,12){\line (1,0){41}}
\put(184,8){\line(1,0){41}}
\put(200,5.5){$<$}
\put(230,10){\circle*{10}}
 
\end{picture}}). According to the table in    (\cite{rub-kyoto} or \cite{Rubenthaler-bouquin-PV})  it has no non-trivial relative invariant, hence it is not $QD1$.

\vskip5pt
\subsubsection{} $S^2(SL(n))\times \C^* (n\geq2)$.

 Up to geometric equivalence this is the representation of $GL(n)$ on $Sym(n)$ given by $g.X=gX{^tg}$ ($g\in GL(n), X\in Sym(n)$). This $PV$ is of commutative parabolic type corresponding to the  diagram \hbox{\unitlength=0.5pt
\begin{picture}(240,30)(-8,5)
\put(10,10){\circle*{10}}
\put(15,10){\line (1,0){30}}
\put(50,10){\circle*{10}}
\put(55,10){\line (1,0){30}}
\put(90,10){\circle*{10}}
\put(95,10){\circle*{1}}
\put(100,10){\circle*{1}}
\put(105,10){\circle*{1}}
\put(110,10){\circle*{1}}
\put(115,10){\circle*{1}}
\put(120,10){\circle*{1}}
\put(125,10){\circle*{1}}
\put(130,10){\circle*{1}}
\put(135,10){\circle*{1}}
\put(140,10){\circle*{10}}
\put(145,10){\line (1,0){30}}
\put(180,10){\circle*{10}}
\put(184,12){\line (1,0){41}}
\put(184,8){\line(1,0){41}}
\put(200,5.5){$<$}
\put(230,10){\circle*{10}}
\put(230,10){\circle{16}}
 \put(260,6){$C_{n} $.}
\end{picture}}

The unique fundamental relative invariant is the determinant, therefore it is $QD1$. It corresponds   to case $(2)$ in Table 2.

\vskip5pt
\subsubsection{} $\Lambda^2(SL(n))\times \C^* (n\geq4)$.

Up to geometric equivalence this is the representation of $GL(n)$ on $AS(n)$ given by $g.X=gX{^tg}$ ($g\in GL(n), X\in AS(n)$). This $PV$ is of commutative parabolic type corresponding to the  diagram \hbox{\hskip -40pt\unitlength=0.5pt
\begin{picture}(340,45)(-82,5)
\put(10,10){\circle*{10}}
\put(15,10){\line (1,0){30}}
\put(50,10){\circle*{10}}
\put(55,10){\line (1,0){30}}
\put(90,10){\circle*{10}}
\put(95,10){\line (1,0){30}}
\put(130,10){\circle*{10}}
\put(140,10){\circle*{1}}
\put(145,10){\circle*{1}}
\put(150,10){\circle*{1}}
\put(155,10){\circle*{1}}
\put(160,10){\circle*{1}}
\put(170,10){\circle*{10}}
\put(175,10){\line (1,0){30}}
\put(210,10){\circle*{10}}
\put(215,14){\line (1,1){20}}
\put(240,36){\circle*{10}}
\put(240,36){\circle{16}}
\put(215,6){\line (1,-1){20}}
\put(240,-16){\circle*{10}}
\put(260,6){$D_{n} $.}
 \end{picture}
\hskip 30pt} 
\vskip 15pt
It is well known that there is no relative invariant if $n$ is odd, and that for $n$ even the  unique fundamental relative invariant is the pfaffian. Therefore it is $QD1$ if and only if $n=2p$.  It corresponds  to case $(3)$ in Table 2.

\vskip5pt
\subsubsection{} $SL(n)\otimes SL(m)^* \times \C^* (n,m\geq2)$.

By Remark \ref{rem-geom-equ}, this representation is geometrically equivalent to case  $SL(n)\otimes SL(m) \times \C^* (n,m\geq2)$ which is considered in Table 1 of Benson-Ratcliff (\cite {Benson-Ratcliff-article}). This is the representation of $SL(n)\times SL(m)$ on the space $M_{n,m}$   given by $(g_{1},g_{2}).X=g_{1}Xg_{2}^{-1}$, $g_{1}\in SL(n), g_{2}\in SL(m), X\in M_{n,m}$. This is a commutative parabolic $PV$ corresponding to the diagram \hbox{\hskip 65pt \unitlength=0.5pt
\hskip-100pt 
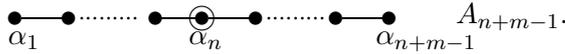
\begin{picture}(400,30)(0,5)
\put(90,10){\circle*{10}}
\put(85,-10){$\alpha_1$}
\put(95,10){\line (1,0){30}}
\put(130,10){\circle*{10}}
 
\put(140,10){\circle*{1}}
\put(145,10){\circle*{1}}
\put(150,10){\circle*{1}}
\put(155,10){\circle*{1}}
\put(160,10){\circle*{1}}
\put(165,10){\circle*{1}}
\put(170,10){\circle*{1}}
\put(175,10){\circle*{1}}
\put(180,10){\circle*{1}}
 \put(195,10){\circle*{10}}
 
\put(195,10){\line (1,0){30}}
\put(230,10){\circle*{10}}
\put(220,-10){$\alpha_{n}$}
\put(230,10){\circle{18}}
\put(235,10){\line (1,0){30}}
\put(270,10){\circle*{10}}
 
\put(280,10){\circle*{1}}
\put(285,10){\circle*{1}}
\put(290,10){\circle*{1}}
\put(295,10){\circle*{1}}
\put(300,10){\circle*{1}}
\put(305,10){\circle*{1}}
\put(310,10){\circle*{1}}
\put(315,10){\circle*{1}}
\put(320,10){\circle*{1}}
\put(330,10){\circle*{10}}
 
\put(335,10){\line (1,0){30}}
\put(370,10){\circle*{10}}
\put(360,-10){$\alpha_{n+m-1}$}
\put(420, 6)  {$A_{n+m-1}$.}
\end{picture} \hskip 50pt
} 
\vskip 15pt

 If $n\neq m$, then there is no fundamental relative invariant. If $n=m$, then the unique fundamental relative invariant is the determinant. Hence it is $QD1$ if and only if $n=m$. It corresponds to case $(4)$ in Table 2.
\vskip5pt
\subsubsection{} $SL(2)\otimes Sp(n) \times \C^* (n\geq2)$. 

Up to geometric equivalence we can consider that this is the representation of $G=GL(2)\times Sp(n)$ acting on $V= M_{2n,2} $ by 
$$(g_{1},g_{2}).X=g_{2}X{^tg_{1}},\,g_{1}\in GL(2), g_{2}\in Sp(n),X\in M_{2n,2}$$This is a regular irreducible $PV$ of parabolic type (not commutative), corresponding to the diagram

\vskip 5pt
\raisebox{8pt}{
\hbox{\unitlength=0.5pt
\hskip 100pt\begin{picture}(300,30)
\put(50,10){\circle*{10}}
\put(55,10){\line (1,0){30}}
\put(90,10){\circle*{10}}
\put(90,10){\circle{16}}
\put(95,10){\line (1,0){30}}
\put(130,10){\circle*{10}}
\put(130,10){\circle*{1}}
\put(135,10){\circle*{1}}
\put(140,10){\circle*{1}}
\put(145,10){\circle*{1}}
\put(150,10){\circle*{1}}
\put(155,10){\circle*{1}}
\put(160,10){\circle*{1}}
\put(165,10){\circle*{1}}
\put(170,10){\circle*{1}}
\put(175,10){\circle*{10}}
\put(180,10){\line (1,0){30}}
\put(215,10){\circle*{10}}
\put(219,12){\line (1,0){41}}
\put(219,8){\line(1,0){41}}
\put(235,5.5){$<$}
\put(265,10){\circle*{10}}
\end{picture} \,\,\,\raisebox{3pt} {$C_{n+2}$}
}}
\vskip 5pt
(see \cite{Sato-Kimura}, \cite{rub-kyoto}, \cite {Rubenthaler-bouquin-PV}). Hence it is $QD1$. According to the computations in \cite{Sato-Kimura} (Proposition $17$ p.100-101), the fundamental relative invariant is $f(X)= Pf({^tX}JX)$, where $X\in M_{2n,2}(\C)$, where $J=\begin{pmatrix}0&Id_{n}\\
-Id_{n}&0
\end{pmatrix}$, and where $Pf(.)$ is the pfaffian of a $2\times 2$  skew symmetric matrix. It is case $(6)$ in Table $2$.

\vskip5pt
\subsubsection{} $SL(3)\otimes Sp(n) \times \C^* (n\geq2)$. 

This is a non regular irreducible $PV$ of parabolic type corresponding to the diagram

\hskip 100pt \hbox{\unitlength=0.5pt
\begin{picture}(240,10)
\put(10,7){\circle*{10}}
\put(15,7){\line (1,0){30}}
\put(50,7){\circle*{10}}
\put(55,7){\line (1,0){30}}
\put(90,7){\circle*{10}}
\put(90,7){\circle{16}}
\put(95,7){\circle*{1}}
\put(100,7){\circle*{1}}
\put(105,7){\circle*{1}}
\put(110,7){\circle*{1}}
\put(115,7){\circle*{1}}
\put(120,7){\circle*{1}}
\put(125,7){\circle*{1}}
\put(130,7){\circle*{1}}
\put(135,7){\circle*{1}}
\put(140,7){\circle*{10}}
\put(145,7){\line (1,0){30}}
\put(180,7){\circle*{10}}
\put(184,10){\line (1,0){41}}
\put(184,4){\line(1,0){41}}
\put(200,2.5){$<$}
\put(230,7){\circle*{10}}
 \end{picture}} 
 
 of type $C_{n+3}$, it is known (\cite{rub-kyoto} or \cite{Rubenthaler-bouquin-PV}) that it has no non-trivial relative invariant, hence it is not $QD1$.

\vskip5pt
\subsubsection{} $SL(n)\otimes Sp(2) \times \C^* (n\geq4)$.   

This  again is an irreducible $PV$ of parabolic type corresponding to the diagram

\hskip 100pt \hbox{\unitlength=0.5pt
\begin{picture}(240,10)
\put(10,7){\circle*{10}}
\put(15,7){\line (1,0){30}}
\put(50,7){\circle*{10}}
\put(55,7){\line (1,0){30}}
\put(90,7){\circle*{10}}
\put(95,7){\circle*{1}}
\put(100,7){\circle*{1}}
\put(105,7){\circle*{1}}
\put(110,7){\circle*{1}}
\put(115,7){\circle*{1}}
\put(120,7){\circle*{1}}
\put(125,7){\circle*{1}}
\put(130,7){\circle*{1}}
\put(135,7){\circle*{1}}
\put(140,7){\circle*{10}}
\put(140,7){\circle{16}}
\put(145,7){\line (1,0){30}}
\put(180,7){\circle*{10}}
\put(184,10){\line (1,0){41}}
\put(184,4){\line(1,0){41}}
\put(200,2.5){$<$}
\put(230,7){\circle*{10}}
 
\end{picture}}, 

of type $C_{n+2}$. It is known (\cite{rub-kyoto} or \cite{Rubenthaler-bouquin-PV}) that this $PV$ has a non-trivial relative invariant if and only if $n=4$. Hence this space is $QD1$ if and only if $n=4$, and then it is regular. In this case the group $SL(4)\times Sp(2)$ acts on $M_{4}$ by $(g_{1},g_{2}).X=g_{1}Xg_{2}^{-1}$, and  the fundamental relative invariant is the determinant. It is case $(7)$ in Table 2.

\vskip5pt
\subsubsection{} $Spin(7)\times \C^*$. 

This space is known (\cite{rub-kyoto} or \cite{Rubenthaler-bouquin-PV}) to be an irreducible regular $PV$ of parabolic type inside $F_{4}$ corresponding to the diagram \raisebox{-2pt}{{
\hbox{\unitlength=0.5pt
\begin{picture}(145,30)
\put(10,10){\circle*{10}}
\put(15,10){\line (1,0){30}}
\put(50,10){\circle*{10}}
\put(55,12){\line (1,0){39}}
\put(65,5.5){$>$}
\put(55,8){\line(1,0){39}}
\put(90,10){\circle*{10}}
\put(95,10){\line (1,0){40}}
\put(140,10){\circle*{10}}
\put(140,10){\circle{16}}
\end{picture} 
}}}. Here the space has dimension $8$ and the action is obtained by embedding $Spin(7)$ into $SO(8)$. The fundamental relative invariant is the nondegenerate quadratic form which defines $SO(8)$. It is case $(8)$ in Table 2.

\vskip5pt
\subsubsection{} $Spin(9)\times \C^*$. 

According to \cite{Sato-Kimura}, p. 146, number (19) of Table, this is an irreducible regular $PV$ whose fundamental relative invariant is a quadratic form, hence it is $QD1$. According to the diagramatical rules in Remark \ref{rem-diagram} d) it is not of parabolic type. But it has nevertheless an interesting connection with $PV$'s of parabolic type, see \cite{Rubenthaler-J-Alg}, Theorem 5.1. p. 377. It is case $(9)$ in Table 2.

\vskip5pt
\subsubsection{} $Spin(10)\times \C^*$. 

This is a $PV$ of parabolic type inside $E_{6}$ (corresponding to the diagram {\hbox{\unitlength=0.5pt
\begin{picture}(200,20)(-10,-7)
\put(10,0){\circle*{10}}
\put(15,0){\line  (1,0){30}}
\put(50,0){\circle*{10}}
\put(55,0){\line  (1,0){30}}
\put(90,0){\circle*{10}}
\put(90,-5){\line  (0,-1){30}}
\put(90,-40){\circle*{10}}
\put(95,0){\line  (1,0){30}}
\put(130,0){\circle*{10}}
\put(135,0){\line  (1,0){30}}
\put(170,0){\circle*{10}}
\put(170,0){\circle{16}}
\end{picture}}).  
\vskip 15pt
According to the table in    (\cite{rub-kyoto} or \cite{Rubenthaler-bouquin-PV})  it has no non-trivial relative 
 invariant, hence it is not $QD1$.

\vskip5pt
\subsubsection{} $G_{2}\times \C^*$.

 According to \cite{Sato-Kimura}, p. 146, number (25) of Table, this is an irreducible regular $PV$ whose fundamental relative invariant is a quadratic form, hence it is $QD1$. According to the diagramatical rules in Remark \ref{rem-diagram} d) it is not of parabolic type. But it has nevertheless an interesting connection with $PV$'s of parabolic type, see \cite{Rubenthaler-J-Alg}, Theorem 6.1. p. 381. It is case $(10)$ in Table 2.

\vskip5pt
\subsubsection{} $E_{6}\times \C^*$. 

This space is known (\cite{rub-kyoto} or \cite{Rubenthaler-bouquin-PV}) to be an irreducible regular $PV$ of parabolic type inside $E_{7}$ corresponding to the diagram \hbox{\hskip-40pt\unitlength=0.5pt
 \begin{picture}(300,35)(-82,-10)
\put(10,0){\circle*{10}}
\put(15,0){\line  (1,0){30}}
\put(50,0){\circle*{10}}
\put(55,0){\line  (1,0){30}}
\put(90,0){\circle*{10}}
\put(90,-5){\line  (0,-1){30}}
\put(90,-40){\circle*{10}}
\put(95,0){\line  (1,0){30}}
\put(130,0){\circle*{10}}
\put(135,0){\line  (1,0){30}}
\put(170,0){\circle*{10}}
\put(175,0){\line (1,0){30}}
\put(210,0){\circle*{10}}
\put(210,0){\circle{18}}
\end{picture}
}.  
 \vskip 15pt
   It is the 
 $27$-dimensional representation of $E_{6}$. Its fundamental relative invariant has degree $3$, it is  known as the Freudenthal cubic. This is the case $(5)$ in Table 2.

\subsection{Non-irreducible MF spaces}\hfill

Here  we examine the cases arising in Table II of \cite{Benson-Ratcliff-article}. We keep the same notations  for the representations  as before. In addition we adopt also the following notation from \cite{Benson-Ratcliff-article}. If $(G_{1},V_{1})$ and $(G_{2},V_{2})$ are representations of two semi-simple groups $G_{1}$ and $G_{2}$ which share a common simple factor $H$, then the notation $G_{1}\oplus_{H}G_{2}$ denotes the image of the representation on $V_{1}\oplus V_{2}$ where the common factor $H$ acts diagonally. For example $SL(n)\oplus_{SL(n)}SL(n)$ denotes the direct sum representation ($SL(n), \C^n\oplus \C^n)$, and $Spin(8)\oplus _{Spin(8)}SO(8)$ denotes the action of $Spin(8)$ on $\C^8\oplus \C^8$ via the Spin representation on the first $\C^8$  factor and via the natural representation of $SO(8)$ on the second factor.

\vskip5pt
\subsubsection{} $(SL(n)\oplus_{SL(n)}SL(n))\times {\C^*}^2 (n\geq 2)$. 

This space is a parabolic $PV$ corresponding to the diagram 

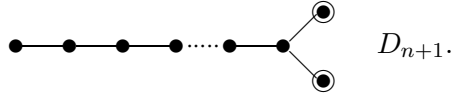
\centerline{\raisebox{-20pt} {\vbox{{
\hbox{\unitlength=0.5pt
\begin{picture}(240,35)(0,-15)
\put(10,0){\circle*{10}}
\put(15,0){\line (1,0){30}}
\put(50,0){\circle*{10}}
\put(55,0){\line (1,0){30}}
\put(90,0){\circle*{10}}
\put(95,0){\line (1,0){30}}
\put(130,0){\circle*{10}}
\put(140,0){\circle*{1}}
\put(145,0){\circle*{1}}
\put(150,0){\circle*{1}}
\put(155,0){\circle*{1}}
\put(160,0){\circle*{1}}
\put(170,0){\circle*{10}}
\put(175,0){\line (1,0){30}}
\put(210,0){\circle*{10}}
\put(215,4){\line (1,1){20}}
\put(240,26){\circle*{10}}
\put(240,26){\circle{16}}
\put(215,-4){\line (1,-1){20}}
\put(240,-26){\circle*{10}}
\put(240,-26){\circle{16}}
\end{picture}\hskip 20pt\raisebox {5pt}{$D_{n+1}.$} }
}}}}
\vskip 20pt
Lets us show that this space is $QD1$ if and only if $n=2$. The representation is given by 
$(g,\lambda,\mu).X= gX\begin{pmatrix}\lambda&0\\
0&\mu
\end{pmatrix}$, where $g\in SL(n), X\in M_{n,2},\lambda,\mu\in \C^*$. Set $X_{0}=\left[\begin{array}{cc}1&0\\
0&1\\
0&0\\
\vdots&\vdots\\
0&0\\
\end{array}\right]$. A simple computation shows that the isotropy subgroup $G_{X_{0}}$ of $X_{0}$ is the subgroup of $SL(n)\times {C^*}^2$ consisting of  triplets of the form  
($\left[
\begin{array}{c|c}
{\begin{matrix}\lambda^{-1}& 0\\
0&\mu^{-1}\\
\end{matrix}} & \beta \\ \hline
0 &\delta
\end{array}\right]
 ,\lambda,\mu$). Then, an easy calculation shows that $\dim G- \dim G_{X_{0}}=2n$, and hence the point $X_{0}$ is generic. Moreover the preceding computation of $G_{X_{0}}$ shows that, if $n>2$,  the subgroup generated by the derived group $SL(n)\times \{1\}\times\{1\}$ and $G_{X_{0}}$ is the full group $SL(n)\times {\C^*}^2$. According to Remark \ref{rem-caracteres} this proves that if $n\neq2$ there exists no non-trivial relative invariant and hence the space is not $QD1$. If $n=2$, the space is the space $M_{2}$  and then the determinant is the only fundamental relative invariant. This is a particular case of $(1)$ in Table 3.

\vskip5pt
\subsubsection{} $(SL(n)^*\oplus_{SL(n)}SL(n))\times {\C^*}^2 (n\geq 2)$.   

   This is the representation of $SL(n)\times  {\C^*}^2$ on $M_{1,2}\times M_{2,1}$ given by $(g,\lambda,\mu)(u,v)= (\lambda ug^{-1},\mu gv)$ where $\lambda,\mu\in \C^*$, $g\in SL(n)$, $u\in M_{1,2}$ and $v\in M_{2,1}$. This a parabolic $PV$ corresponding to the diagram
\raisebox{-3pt}{{
\hbox{\unitlength=0.5pt
\begin{picture}(250,30)
\put(10,10){\circle*{10}}
\put(10,10){\circle{16}}
\put(15,10){\line (1,0){30}}
\put(50,10){\circle*{10}}
\put(60,10){\circle*{1}}
\put(65,10){\circle*{1}}
\put(70,10){\circle*{1}}
\put(75,10){\circle*{1}}
\put(80,10){\circle*{1}}
\put(90,10){\circle*{10}}
\put(95,10){\line (1,0){30}}
\put(130,10){\circle*{10}}
 \put(135,10){\line (1,0){30}}
\put(170,10){\circle*{10}}
\put(180,10){\circle*{1}}
\put(185,10){\circle*{1}}
\put(190,10){\circle*{1}}
\put(195,10){\circle*{1}}
\put(200,10){\circle*{1}}
\put(210,10){\circle*{10}}
\put(215,10){\line (1,0){30}}
\put(250,10){\circle*{10}}
\put(250,10){\circle{16}}
\end{picture}
}}\hskip 6pt \raisebox{3pt}{$A_{n+1}$}}
It is easily seen   that the quadratic form $f(u,v)= uv$ is the unique fundamental relative invariant and as the generic isotropy subgroup is reductive, this $PV$ is regular. This is a particular case of a family of  so-called $Q$-irreducible $PV$'s of parabolic type. We refer the reader interested into details to Lemma 4.8 in \cite{Rubenthaler-decomposition}.  It is case $(1)$ in Table 3.

\vskip5pt
\subsubsection{} $(SL(n)\oplus_{SL(n)}\Lambda^2(SL(n)) \times {\C^*}^2 (n\geq 4)$.  

The representation is given by $(g,\lambda,\mu).(u,x)=( \lambda g u, \mu g x {^t\hskip -2pt g}$) where $\lambda,\mu\in \C^*$, $g\in SL(n)$, $u\in \C^n, x\in AS(n)$. This $PV$ is not of parabolic type except for the cases where $n=5,6,7$ which correspond respectively to the following diagrams:
\vskip 5pt

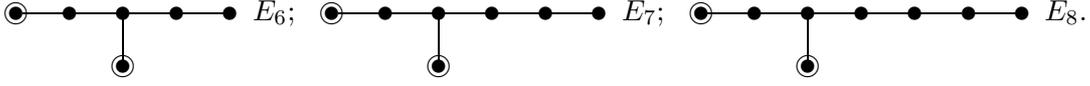
 \hbox{\unitlength=0.5pt
\begin{picture}(190,40)(-10,-45)
\put(10,0){\circle*{10}}
\put(10,0){\circle{16}}
\put(15,0){\line  (1,0){30}}
\put(50,0){\circle*{10}}
\put(55,0){\line  (1,0){30}}
\put(90,0){\circle*{10}}
\put(90,-5){\line  (0,-1){30}}
\put(90,-40){\circle*{10}}
\put(90,-40){\circle{16}}
\put(95,0){\line  (1,0){30}}
\put(130,0){\circle*{10}}
\put(135,0){\line  (1,0){30}}
\put(170,0){\circle*{10}}
\end{picture} \hskip 0pt\raisebox {20pt}{$E_{6}$;}
\vbox{\hbox{\unitlength=0.5pt
\begin{picture}(230,40)(-10,-45)
\put(10,0){\circle*{10}}
\put(10,0){\circle{16}}
\put(15,0){\line  (1,0){30}}
\put(50,0){\circle*{10}}
\put(55,0){\line  (1,0){30}}
\put(90,0){\circle*{10}}
\put(90,-5){\line  (0,-1){30}}
\put(90,-40){\circle*{10}}
\put(90,-40){\circle{16}}
\put(95,0){\line  (1,0){30}}
\put(130,0){\circle*{10}}
\put(135,0){\line  (1,0){30}}
\put(170,0){\circle*{10}}
\put(175,0){\line (1,0){30}}
\put(210,0){\circle*{10}}
\end{picture} \hskip 0pt\raisebox {20pt}{$E_{7}$;}
}}\vbox{\hbox{\unitlength=0.5pt
\begin{picture}(270,40)(-10,-45)
\put(10,0){\circle*{10}}
\put(10,0){\circle{16}}
\put(15,0){\line  (1,0){30}}
\put(50,0){\circle*{10}}
\put(55,0){\line  (1,0){30}}
\put(90,0){\circle*{10}}
\put(90,-5){\line  (0,-1){30}}
\put(90,-40){\circle*{10}}
\put(90,-40){\circle{16}}
\put(95,0){\line  (1,0){30}}
\put(130,0){\circle*{10}}
\put(135,0){\line  (1,0){30}}
\put(170,0){\circle*{10}}
\put(175,0){\line (1,0){30}}
\put(210,0){\circle*{10}}
\put(215,0){\line (1,0){30}}
\put(250,0){\circle*{10}}
\end{picture} \hskip 0pt\raisebox {20pt}{$E_{8}$.}
}}
}  
\vskip 8pt
 
$\bullet$ Suppose first that $n=2p$ is even.    It is well known  that the restriction of the representation to $AS(n)$   is  a  regular $PV$ of parabolic type, and that its unique fundamental relative invariant is the pfaffian. Moreover the generic isotropy subgroup of this component at the point $J= \left[
\begin{array}{c|c}
{ 0} &Id_{p}\\ \hline
-Id_{p} &0
\end{array}\right]$ is the  subgroup $ Sp(p)\times \C^*$.  The restriction of the "natural" representation of $SL(2p)\times \C^*$ on $\C^{2p}$ to $Sp(p)\times \C^*$ is well known to have no non-trivial relative invariant. Hence, by Proposition \ref{prop-isotropie-rang1} this space is $QD1$. As the fundamental relative invariant does not depend on all variables, it is not regular. It is case $(2)(a)$ in Table 3.
\vskip 5pt
  $\bullet$ Suppose now that $n=2p+1$ is odd. Rather than the group $SL(n)\times {\C^*}^2$, we will here consider the Lie algebra ${\go g}={\go{ gl}}(n)\times \C$ acting on $V= {\C^n}\oplus AS(n)$ by $(U,\lambda)(v,x)=(\lambda v+Uv, Ux+x{^tU)} $ where $\lambda\in \C,U\in {\go{ gl}}(n), v\in {\C^n}, x\in AS(n) $.  Once we identify ${\go{ gl}}(n)$ with ${\go{ sl}}(n)\times \C$ this is essentially the derived representation of $(SL(n)\oplus_{SL(n)}\Lambda^2(SL(n)) \times {\C^*}^2$. Consider the point $X_{0}=(v_{0},x_{0})\in {\C^n}\oplus AS(n)$ where $v_{0}=\left[
\begin{array}{c}
0\\
\vdots\\
0\\
1
\end{array}\right]$ and where $x_{0}= \left[
\begin{array}{c|c}
{ J} &0\\ \hline
0 &0
\end{array}\right]$. An easy computation shows that the isotropy subalgebra ${\go g}_{X_{0}}$ is given by 
$${\go g}_{X_{0}}=\{( \left[
\begin{array}{c|c}
A &0\\ \hline
0 &-\lambda
\end{array}\right],\lambda ), \lambda\in \C, A\in Sp(p)\}.$$
As $\dim {\go g}-\dim {\go g}_{X_{0}}=\dim V$,  the point $X_{0}$ is generic. The Lie algebra ${\go g}_{X_{0}}$ is the Lie algebra of a reductive subgroup. Hence this space is regular. As $[{\go g},{\go g}]= {\go{sl}}(n)$, the Lie algebra generated by ${\go g}_{X_{0}}$ and $[{\go g},{\go g}]$ is equal to $\go{sl}(n)$, and hence ${\go g}/{\go{sl}}(n)\simeq \C$. According   to  Remark \ref{rem-caracteres}, there exists exactly one (up to constants)  fundamental relative invariant, and hence this space is QD1. Keeping the same notations as above it is easy to see that the polynomial 

$$f(v,x)= Pf(\left[
\begin{array}{c|c}
x &v\\ \hline
-^tv &0
\end{array}\right]), v\in {\C^n}, x\in AS(2p+1)$$
is a fundamental relative invariant. These spaces are $Q$-irreducible in the sense of \cite{Rubenthaler-decomposition} (see Remark 4.15 in \cite{Rubenthaler-decomposition}).
It is case $(2)(b)$ in Table 3.

\vskip5pt
\subsubsection{} $(SL(n)^*\oplus_{SL(n)}\Lambda^2(SL(n)) \times {\C^*}^2 (n\geq 4)$. 

This $PV$ is always of parabolic type. The corresponding diagram is the following:

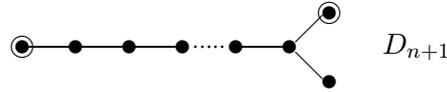
\centerline{\raisebox{-20pt} {\vbox{{
\hbox{\unitlength=0.5pt
\begin{picture}(240,35)(0,-15)
\put(10,0){\circle*{10}}
\put(10,0){\circle{16}}
\put(15,0){\line (1,0){30}}
\put(50,0){\circle*{10}}
\put(55,0){\line (1,0){30}}
\put(90,0){\circle*{10}}
\put(95,0){\line (1,0){30}}
\put(130,0){\circle*{10}}
\put(140,0){\circle*{1}}
\put(145,0){\circle*{1}}
\put(150,0){\circle*{1}}
\put(155,0){\circle*{1}}
\put(160,0){\circle*{1}}
\put(170,0){\circle*{10}}
\put(175,0){\line (1,0){30}}
\put(210,0){\circle*{10}}
\put(215,4){\line (1,1){20}}
\put(240,26){\circle*{10}}
\put(240,26){\circle{16}}
\put(215,-4){\line (1,-1){20}}
\put(240,-26){\circle*{10}}
\end{picture}\hskip 20pt\raisebox {5pt}{$D_{n+1}$} }
}}}}
\vskip 20pt
Up to geometric equivalence we can take here $G=GL(n)\times {\C^*}$ acting on    $V= M_{1,n}\oplus AS(n)$  by $(g,\lambda).(v,x)=(\lambda vU^{-1},Ux\,{^tU})$. 

\vskip 5pt
$\bullet $  Suppose first that $n=2p$ is even. The restriction of the representation to $AS(n)$ is a regular $PV$ whose fundamental relative invariant is the pfaffian. The partial isotropy of $J= \left[
\begin{array}{c|c}
{ 0} &Id_{p}\\ \hline
-Id_{p} &0
\end{array}\right]\in AS(n)$ is equal to $Sp(p)\times \C^*$, and it is well known that the action of $Sp(p)\times \C^*$ on $M_{1,n}$ has non non-trivial relative invariant. Therefore, from Proposition \ref{prop-isotropie-rang1}, we obtain that this $MF$ space is $QD1$. As the fundamental relative invariant does not depend on all variables, it is not regular. It is case $(3)$ in Table 3.

\vskip 5pt

$\bullet$ Suppose now that $n=2p+1$ is odd.   Rather than the group action, we will  consider here the infinitesimal action. In other words we consider the Lie algebra ${\go g}={\go{ gl}}(n)\times \C$ acting on $V= {M_{1,2p+1}}\oplus AS(2p+1)$ by $(U,\lambda)(v,x)=(\lambda v-vU, Ux+x{^tU)} $ where $\lambda\in \C,U\in {\go{ gl}}(n), v\in {\C^n}, x\in AS(n) $. Consider the element $X_{0}=(v_{0},x_{0})\in V$ which is defined by $v_{0}=(1,0,\dots,0)$ and $x_{0}=  \left[
\begin{array}{c|c}
{ J} &0\\ \hline
0 &0
\end{array}\right]\in AS(2p+1)$ with $J= \left[
\begin{array}{c|c}
{ 0} &Id_{p}\\ \hline
-Id_{p} &0
\end{array}\right]$. A computation shows that the isotropy subalgebra ${\go g}_{X_{0}}$ is the set of couples of the form 
$(\left[
\begin{array}{c|c}
A &B\\ \hline
0 &D
\end{array}\right], \lambda) $,

-- where $A=\left[
\begin{array}{c|c}
\alpha&\beta\\ \hline
 \gamma&-^t\alpha
\end{array}\right]$ with  $\alpha=\left[\begin{array}{c}
\lambda, 0\dots0\\
A_{1}\end{array}\right]$, $A_{1}\in M_{p-1,p}$; $\beta=\left[\begin{array}{cc}0&0\\
0& b\end{array}\right]$, $b\in Sym(p-1)$; $\gamma\in Sym(p)$.

-- where $B=\left[\begin{array}{c}0\\
\tilde B\end{array}\right]$, $\tilde B\in \C^{2p-1}$

-- and where $D,\lambda \in \C$.

Then one verifies  that $\dim {\go g}-\dim {\go g}_{X_{0}}=\dim V$, and hence $X_{0}$ is generic. Moreover the Lie subalgebra generated by ${\go g}_{X_{0}}$ and $[{\go g} ,{\go g}]=\go{sl}(n)\times \{0\}$ is equal to ${\go g}$. According again to Remark \ref{rem-caracteres}, this shows that there is no non-trivial relative invariant, and therefore this space is not $QD1$.

\vskip 5pt

\vskip5pt
\subsubsection{} $(SL(n)\oplus_{SL(n)} (SL(n)\otimes SL(m)) \times {\C^*}^2 (n,m\geq 2)$. 

It is convenient here to replace this representation by the representation $(G,V)$ where $G=GL(n)\times GL(m)$ acts on $V=M_{n,1}\oplus M_{n,m}$ by
$$(g_{1},g_{2})(v,x)=(g_{1}v,g_{1}xg_{2}^{-1}), g_{1}\in GL(n),g_{2}\in GL(m),v\in M_{n,1},x\in M_{n,m}.$$
Due to   Remark \ref{rem-geom-equ}, this representation is geometrically equivalent to the first one. It is not of parabolic type except for the following cases:

\vskip 8pt

  {\raisebox{5pt} {$n=3,m\in \N:$}{\vbox{{
\hbox{\unitlength=0.5pt
\begin{picture}(240,35)(0,-15)
\put(10,0){\circle*{10}}
\put(15,0){\line (1,0){30}}
\put(50,0){\circle*{10}}
\put(55,0){\line (1,0){30}}
\put(90,0){\circle*{10}}
\put(100,0){\circle*{1}}
\put(105,0){\circle*{1}}
\put(110,0){\circle*{1}}
\put(115,0){\circle*{1}}
\put(115,0){\circle*{1}}
\put(115,0){\circle*{1}}
\put(120,0){\circle*{1}}
\put(130,0){\circle*{10}}
\put(135,0){\line (1,0){30}}
\put(170,0){\circle*{10}}
\put(170,0){\circle{16}}
\put(175,0){\line (1,0){30}}
\put(210,0){\circle*{10}}
\put(215,4){\line (1,1){20}}
\put(240,26){\circle*{10}}
\put(240,26){\circle{16}}
\put(215,-4){\line (1,-1){20}}
\put(240,-26){\circle*{10}}
\end{picture}\hskip 20pt\raisebox {5pt}{$D_{m+3}$} }
}}}}

\vskip 20pt

\raisebox {20pt}{$n=4,m=2:$} \hbox{ \unitlength=0.5pt
\begin{picture}(190,40)(-10,-45)
\put(10,0){\circle*{10}}
\put(15,0){\line  (1,0){30}}
\put(50,0){\circle*{10}}
\put(55,0){\line  (1,0){30}}
\put(90,0){\circle*{10}}
\put(90,-5){\line  (0,-1){30}}
\put(90,-40){\circle*{10}}
\put(90,-40){\circle{16}}
\put(95,0){\line  (1,0){30}}
\put(130,0){\circle*{10}}
\put(130,0){\circle{16}}
\put(135,0){\line  (1,0){30}}
\put(170,0){\circle*{10}}
\end{picture}  \hskip 0pt\raisebox {20pt}{$E_{6}$}
}}

\vskip 5pt
 
\raisebox {20pt}{$n=4,m=3:$}{\hbox{\unitlength=0.5pt
\begin{picture}(230,40)(-10,-45)
\put(10,0){\circle*{10}}
\put(15,0){\line  (1,0){30}}
\put(50,0){\circle*{10}}
\put(55,0){\line  (1,0){30}}
\put(90,0){\circle*{10}}
\put(90,-5){\line  (0,-1){30}}
\put(90,-40){\circle*{10}}
\put(90,-40){\circle{16}}
\put(95,0){\line  (1,0){30}}
\put(130,0){\circle*{10}}
\put(130,0){\circle{16}}
\put(135,0){\line  (1,0){30}}
\put(170,0){\circle*{10}}
\put(175,0){\line (1,0){30}}
\put(210,0){\circle*{10}}
\end{picture}\hskip 5pt\raisebox {20pt}{$E_{7};$} \hskip 30pt

\vbox{\raisebox {20pt}{$n=5,m=2:$}{\hbox{\unitlength=0.5pt
\begin{picture}(230,40)(-10,-45)
\put(10,0){\circle*{10}}
\put(15,0){\line  (1,0){30}}
\put(50,0){\circle*{10}}
\put(50,0){\circle{16}}
\put(55,0){\line  (1,0){30}}
\put(90,0){\circle*{10}}
\put(90,-5){\line  (0,-1){30}}
\put(90,-40){\circle*{10}}
\put(90,-40){\circle{16}}
\put(95,0){\line  (1,0){30}}
\put(130,0){\circle*{10}}
\put(135,0){\line  (1,0){30}}
\put(170,0){\circle*{10}}
\put(175,0){\line (1,0){30}}
\put(210,0){\circle*{10}}
\end{picture} \hskip 0pt\raisebox {20pt}{$E_{7}$}
}}}}

\vskip 10pt

\raisebox {20pt}{$n=4,m=4:$}\kern-5pt\hbox{\unitlength=0.5pt
\begin{picture}(270,40)(-10,-45)
\put(10,0){\circle*{10}}
\put(15,0){\line  (1,0){30}}
\put(50,0){\circle*{10}}
 \put(55,0){\line  (1,0){30}}
\put(90,0){\circle*{10}}
\put(90,-5){\line  (0,-1){30}}
\put(90,-40){\circle*{10}}
\put(90,-40){\circle{16}}
\put(95,0){\line  (1,0){30}}
\put(130,0){\circle*{10}}
\put(130,0){\circle{16}}
\put(135,0){\line  (1,0){30}}
\put(170,0){\circle*{10}}
\put(175,0){\line (1,0){30}}
\put(210,0){\circle*{10}}
\put(215,0){\line (1,0){30}}
\put(250,0){\circle*{10}}
\end{picture}  \raisebox {20pt}{$E_{8};$

\hskip 5pt
\vbox{{$n=6,m=2:$}\kern-5pt\raisebox{-20pt}{\unitlength=0.5pt
\begin{picture}(270,40)(-10,-45)
\put(10,0){\circle*{10}}
\put(15,0){\line  (1,0){30}}
\put(50,0){\circle*{10}}
\put(50,0){\circle{16}}
\put(55,0){\line  (1,0){30}}
\put(90,0){\circle*{10}}
\put(90,-5){\line  (0,-1){30}}
\put(90,-40){\circle*{10}}
\put(90,-40){\circle{16}}
\put(95,0){\line  (1,0){30}}
\put(130,0){\circle*{10}}
\put(135,0){\line  (1,0){30}}
\put(170,0){\circle*{10}}
\put(175,0){\line (1,0){30}}
\put(210,0){\circle*{10}}
\put(215,0){\line (1,0){30}}
\put(250,0){\circle*{10}}
\end{picture} \hskip 0pt\raisebox {20pt}{$E_{8}$}
}}}}
 
  \vskip 5pt

\vskip 5pt

$\bullet$  Suppose $n=m$.  Then  the component $M_{n,n}=M_{n}$ has a unique fundamental relative invariant, namely the determinant. The point $X_{0}=(e_{1},Id_{n})$, where $e_{1}$ is the first vector of the canonical basis of $M_{n,1}\simeq \C^{n}$, is generic. And the partial isotropy subgroup $G_{(0,Id_{n})}$ is the diagonal subgroup $\{(g,g)\in GL(n)\times GL(n)
\}$. Therefore the action of $G_{(0,Id_{n})}$ on $M_{n,1}$ has no relative invariant. According to Proposition \ref{prop-isotropie-rang1}  this space is $QD1$ in the case $m=n$. As the fundamental relative invariant does not depend on all variables, it is not regular. It is case $(4)(a)$ in Table 3.
\vskip 5pt 
$\bullet$ Supose  that $n<m$. A simple calculation shows that the point $X_{0}=(e_{1},x_{0})$   where $x_{0}=\left[\begin{array}{c|c}Id_{n}&0\end{array}\right]$ is generic and that its isotropy subgroup $G_{X_{0}}$ is the set of pairs of matrices of the form $(A, \left[\begin{array}{c|c}A&0\\ \hline
B&C\end{array}\right])$, where $B\in M_{m-n,n}, C\in GL(m-n)$, and where $A=\left[\begin{array}{c|c}1&A_{1}\\ \hline 0&A_{2}\end{array}\right]$, with $A_{1}\in M_{1,n-1}$ and $A_{2}\in GL(n-1)$.   This implies that the subgroup of $G=GL(n)\times GL(m)$ generated by $G_{X_{0}}$ and  the derived group $SL(n)\times SL(m)$ is $G$ itself. Hence from Remark \ref{rem-caracteres}, we know that there is no non-trivial relative invariant, and therefore it is not $QD1$.

\vskip 5pt

$\bullet$  Suppose  that $n>m+1$. Then the element $X_{0}=(e_{n}, x_{0})$ where $x_{0}= \left[\begin{array}{c}Id_{m}\\0\end{array}\right]$ is generic and the isotropy subgroup $G_{X_{0}}$ is the set of pairs of matrices of the form $(\left[\begin{array}{c|c}A&B\\ \hline
0&C\end{array}\right],A)$ where $A\in GL(m)$, where  $B\in M_{m,n-m}$ is of the form $B=\left[\begin{array}{c|c}B'&0\end{array}\right]$ with $B'\in M_{m,n-m-1}$, and where $C\in GL(n-m)$ is of the form $C=\left[\begin{array}{c|c}C_{1}&0\\ \hline C_{2}&1\end{array}\right]$ with $C_{1}\in GL(n-m-1)$ and $C_{2}\in M_{1,n-m-1}$.    Again this implies that the subgroup generated  by $G_{X_{0}}$ and $[G,G]$ is equal to $G$. Hence by Remark \ref{rem-caracteres}, we obtain that this space has no non-trivial relative invariant, and hence it is not $QD1$.

\vskip 5pt
$\bullet$ Suppose finally that $n=m+1$. Then the same calculation as before holds. But now as $n-m=1$, the isotropy subgroup $G_{X_{0}}$ is the set of pairs of matrices of the form $(\left[\begin{array}{c|c}A&0\\ \hline
0&1\end{array}\right],A)$ where $A\in GL(m)$. The subgroup   $\widetilde G$ generated  by $G_{X_{0}}$ and $[G,G]$ is equal to  $\{(g_{1},g_{2})\in G\,|\, \det(g_{1})=\det(g_{2})\}$. This implies that $\dim G/\widetilde G=1$ and hence by Remark \ref{rem-caracteres}, we obtain that this space has one fundamental relative invariant, and therefore it is $QD1$. As the generic isotropy subgroup is reductive, it is regular. It is easy to see that $f(v,x)=\det(v;x)$, where $(x;v)$ is the $n\times n$ matrix obtained by putting the column vector $v$ left to the $m\times n$ matrix $x$, is the fundamental relative invariant. It is case $(4)(b)$ in Table 3.
\vskip5pt
\subsubsection{} $(SL(n)^*\oplus_{SL(n)} (SL(n)\otimes SL(m)) \times {\C^*}^2 (n\geq 3,m\geq 2)$. 

\vskip 5pt
It is convenient here to consider  the representation $(G,V)$ where $G=GL(n)\times GL(m)$ acts on $V=M_{1,n}\oplus M_{n,m}$   by 
$$(g_{1},g_{2})(v,x)=(vg_{1}^{-1},g_{1}xg_{2}^{-1}), g_{1}\in GL(n),g_{2}\in GL(m),v\in M_{1,n},x\in M_{n,m}.$$
This representation is geometrically equivalent to the original one.
This $PV$ is of parabolic type and corresponds to the diagram:

 \raisebox{8pt} {\hbox{\unitlength=0.5pt
\hskip-100pt 
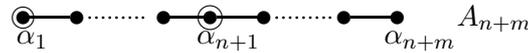
\begin{picture}(700,30)(-300,10)
\put(90,10){\circle*{10}}
\put(90,10){\circle{16}}
\put(85,-10){$\alpha_1$}
\put(95,10){\line (1,0){30}}
\put(130,10){\circle*{10}}
 \put(140,10){\circle*{1}}
\put(145,10){\circle*{1}}
\put(150,10){\circle*{1}}
\put(155,10){\circle*{1}}
\put(160,10){\circle*{1}}
\put(165,10){\circle*{1}}
\put(170,10){\circle*{1}}
\put(175,10){\circle*{1}}
\put(180,10){\circle*{1}}
 \put(195,10){\circle*{10}}
 \put(195,10){\line (1,0){30}}
\put(230,10){\circle*{10}}
\put(220,-10){$\alpha_{n+1}$}
\put(230,10){\circle{18}}
\put(235,10){\line (1,0){30}}
\put(270,10){\circle*{10}}
 \put(280,10){\circle*{1}}
\put(285,10){\circle*{1}}
\put(290,10){\circle*{1}}
\put(295,10){\circle*{1}}
\put(300,10){\circle*{1}}
\put(305,10){\circle*{1}}
\put(310,10){\circle*{1}}
\put(315,10){\circle*{1}}
\put(320,10){\circle*{1}}
\put(330,10){\circle*{10}}
\put(335,10){\line (1,0){30}}
\put(370,10){\circle*{10}}
\put(360,-10){$\alpha_{n+m}$}
\end{picture} 
}}  \raisebox{5pt}{$A_{n+m}$}
\vskip 10pt
 The element $(e_{1},x_{0})$ where $e_{1}=(1,0,\dots,0)$ and where $x_{0}= Id_{n}$ if $n=m$, $x_{0}= \left[\begin{array}{c|c}Id_{n}&0\end{array}\right]$ if $n<m$, and $x_{0}=\left[\begin{array}{c}Id_{m}\\ 0\end{array}\right]$ if $n>m$, is generic and almost the same calculations as in $4.2.5$ show that this $MF$ space is $QD1$ if and only if $n=m$. Moreover by the same argument  it is not a regular $PV$. It is case $(5)$ in Table 3.

\vskip5pt
\subsubsection{} $(SL(2)\oplus_{SL(2)} (SL(2)\otimes Sp(n)) \times {\C^*}^2 (n \geq 2)$. 
\vskip 5pt
This $PV$ is of parabolic type and corresponds to the diagram:
\vskip 5pt
\raisebox{8pt}{
\hbox{\unitlength=0.5pt
\hskip 100pt\begin{picture}(300,30)
\put(10,10){\circle*{10}}
\put(10,10){\circle{16}}
\put(15,10){\line (1,0){30}}
\put(50,10){\circle*{10}}
\put(55,10){\line (1,0){30}}
\put(90,10){\circle*{10}}
\put(90,10){\circle{16}}
\put(95,10){\line (1,0){30}}
\put(130,10){\circle*{10}}
\put(130,10){\circle*{1}}
\put(135,10){\circle*{1}}
\put(140,10){\circle*{1}}
\put(145,10){\circle*{1}}
\put(150,10){\circle*{1}}
\put(155,10){\circle*{1}}
\put(160,10){\circle*{1}}
\put(165,10){\circle*{1}}
\put(170,10){\circle*{1}}
\put(175,10){\circle*{10}}
\put(180,10){\line (1,0){30}}
\put(215,10){\circle*{10}}
\put(219,12){\line (1,0){41}}
\put(219,8){\line(1,0){41}}
\put(235,5.5){$<$}
\put(265,10){\circle*{10}}
\end{picture} \,\,\,\raisebox{3pt} {$C_{n+3}$}
}}

It is convenient to  consider here the group $G=\C^*\times GL(2)\times Sp(n)$ which acts on $V= M_{1,2}\oplus  M_{2n,2}$ by 
$$(\lambda,g_{1},g_{2}).(v,x)= (\lambda v{^t\kern-1pt g_{1}}, g_{2}x\,{^tg_{1}}), \text{where } \lambda\in \C^*,g_{1}\in GL(2), g_{2}\in Sp(n), v\in M_{1,2}, x\in M_{2n,2}.$$ This space is geometrically equivalent to the original one. 
The action of $G$ on $V_{2}$ is a  regular parabolic $PV$ corresponding to the subdiagram  \vskip 5pt
\raisebox{8pt}{
\hbox{\unitlength=0.5pt
\hskip 100pt\begin{picture}(300,30)
\put(50,10){\circle*{10}}
\put(55,10){\line (1,0){30}}
\put(90,10){\circle*{10}}
\put(90,10){\circle{16}}
\put(95,10){\line (1,0){30}}
\put(130,10){\circle*{10}}
\put(130,10){\circle*{1}}
\put(135,10){\circle*{1}}
\put(140,10){\circle*{1}}
\put(145,10){\circle*{1}}
\put(150,10){\circle*{1}}
\put(155,10){\circle*{1}}
\put(160,10){\circle*{1}}
\put(165,10){\circle*{1}}
\put(170,10){\circle*{1}}
\put(175,10){\circle*{10}}
\put(180,10){\line (1,0){30}}
\put(215,10){\circle*{10}}
\put(219,12){\line (1,0){41}}
\put(219,8){\line(1,0){41}}
\put(235,5.5){$<$}
\put(265,10){\circle*{10}}
\end{picture} \,\,\,\raisebox{3pt} {$C_{n+2}$}
}}

(see \cite{Sato-Kimura}, \cite{rub-kyoto}, \cite {Rubenthaler-bouquin-PV}). 
As we have already seen  in section 4.1.7. its fundamental relative invariant is the function $x\longmapsto Pf( {^tx}Jx)$ where $J=\left[ \begin{array}{cc}0&Id_{n}\\-Id_{n}&0\end{array}\right]$.
We know from \cite{Sato-Kimura} (p. 100-101)  that the partial isotropy subalgebra of $({\go g},V_{2})$ corresponding to a certain generic element $x_{0}$ in $V_{2}$ is given by 
$${\go g}_{x_{0}}= (\lambda, -\left[\begin{array}{cc}A_{1}&C_{1}\\
B_{1}&-A_{1}\end{array}\right], \left[\begin{array}{c|c}\begin{array}{cc}A_{1}&0\\
0&A_{4}\end{array}&\begin{array}{cc}B_{1}&0\\
0&B_{4}\end{array}\\ \hline
\begin{array}{cc}C_{1}&0\\
0&C_{4}\end{array}& \begin{array}{cc}-A_{1}&0\\
0&-^tA_{4}\end{array}
\end{array}\right]).$$
where $\lambda,  A_{1}, 
B_{1}, C_{1}\in \C, A_{4}\in {\go {gl}}(n-1), B_{4}, C_{4}\in Sym(n-1)$. This shows that ${\go g}_{x_{0}}\simeq  \C\times {\go {sl}}({2})\times  {\go {sp}}(n-1)$. The action of ${\go g}_{x_{0}}$ on $M_{2,1}$ is then essentially the natural action of  ${\go {gl}}(2)$ on $\C^2$, which is known to have no non trivial relative invariant. Therefore, using again Proposition \ref{prop-isotropie-rang1}, we obtain that this space is $QD1$.  Its fundamental relative invariant is given by $f(v,x)= Pf( {^tx}Jx)$. As this function depends only on $x$, the corresponding $PV$ is not regular. It is case $(6)$ in Table 3.

 \vskip5pt
\subsubsection{} $(SL(n)\otimes SL(2))\oplus_{SL(2)}(SL(2)\otimes SL(m)) \times {\C^*}^2,\, (n,m\geq 2)$. 
\vskip 5pt
This space again is a $PV$ of parabolic type corresponding to the diagram
\vskip 5pt

{
\hbox{\unitlength=0.5pt
\hskip 100pt\begin{picture}(250,30)
\put(10,10){\circle*{10}}
\put(5,-10){$\alpha_1$}
 \put(15,10){\line (1,0){30}}
\put(50,10){\circle*{10}}
\put(60,10){\circle*{1}}
\put(65,10){\circle*{1}}
\put(70,10){\circle*{1}}
\put(75,10){\circle*{1}}
\put(80,10){\circle*{1}}
\put(90,10){\circle*{10}}
\put(95,10){\line (1,0){30}}
\put(130,10){\circle*{10}}
\put(130,10){\circle{16}}
\put(125,-10){$\alpha_n$}
 \put(135,10){\line (1,0){30}}
\put(170,10){\circle*{10}}
\put(175,10){\line (1,0){30}}
 \put(210,10){\circle*{10}}
\put(210,10){\circle{16}}
\put(200,-10){$\alpha_{n+2}$}
\put(215,10){\line (1,0){30}}
 \put(250,10){\circle*{10}}
  \put(260,10){\circle*{1}}
 \put(265,10){\circle*{1}}
\put(270,10){\circle*{1}}
\put(275,10){\circle*{1}}
\put(280,10){\circle*{1}}
\put(290,10){\circle*{10}}
\put(295,10){\line (1,0){30}}
\put(330,10){\circle*{10}}
\put(310,-10){$\alpha_{n+m+1}$}
\end{picture} 
\hskip 70pt\raisebox{3pt}{$A_{n+m+1}$}}

 \vskip 10pt
 
Up to geometric equivalence we can take here   $G= GL(n)\times SL(2)\times GL(m)$ acting on $V=V_{1}\oplus V_{2}$  where $V_{1}=M_{n,2}$ and $V_{2}=M_{2,m}$ by
$$(g_{1},g_{2},g_{3})(u,v)=(g_{1}ug_{2}^{-1},g_{2}vg_{3}^{-1}),\,\,  g_{1}\in GL(n),\, g_{2}\in SL(2),\, g_{3}\in GL(m).$$

\vskip 5pt
a) Let us consider first the case where $n=2$ and $m> 2$ (or equivalently $m=2$ and $n>2$). 

In this case the action of $G$ on $V_{1}=M_{2,2}$ has a non trivial relative invariant (the determinant), the (partial) generic isotropy of the matrix $Id_{2}$ is given by  $G_{Id_{2}}=\{(g,g,g_{3}),g\in SL(2), g_{3}\in GL(m)\}$, and the action of $G_{Id_{2}}$ on $V_{2}=M_{2,m}$ is well known to have no non-trivial relative invariant.
We deduce from Proposition \ref{prop-isotropie-rang1} that this $MF$ space is $QD1$. As the fundamental relative invariant which is given by $f(u,v)=\det u$ depends only on the $V_{1}$ variable, it is not regular. It is case $(7)$ in Table 3.
\vskip 5pt

b) Consider now  the case where $n=m=2$. In this case there are obviously two fundamental relative invariants given by $\det u$ and $\det v$ , $u\in V_{1},v\in V_{2}$. Hence this $MF$ space is not $QD1$.

\vskip 5pt

c) Consider finally  the case where $n\geq m>2$ (or equivalently the case where $m\geq n> 2$).

Define $x_{0}=\left[\begin{array}{c}Id_{2}\\ \hline 0\end{array}\right]\in M_{n,2}$ and 
$y_{0}=\left[\begin{array}{c|c}Id_{2}& 0\end{array}\right]\in M_{2,m}$. The pair $(x_{0},y_{0})$ is a generic element and the corresponding isotropy subgroup is given by

$G_{(x_{0},y_{0})}=$

$$\{(\left[\begin{array}{cc} g_{2}&\beta\\
0&\delta\end{array}\right],g_{2},\left[\begin{array}{cc} g_{2}&0\\
c&d\end{array}\right])\in G\,|\,  g_{2}\in SL(2),\delta \in GL(n-2),\beta \in M_{2,n-2}, d\in GL(m-2,)\}.$$
It is then clear that the subgroup generated by the derived group $SL(n)\times SL(2)\times SL(m)$ and $G_{(x_{0},y_{0})}$ is equal to $G=GL(m)\times SL(2)\times GL(m)$.  Remark \ref{rem-caracteres}  implies   that this space has no non-trivial relative invariant, and hence it is not $QD1$.

 \vskip5pt
\subsubsection{} $(SL(n)\otimes SL(2))\oplus_{SL(2)}(SL(2)\otimes Sp(m)) \times {\C^*}^2,\, (n,m\geq 2)$. 
\vskip 5pt
This space also is a $PV$ of parabolic type corresponding to the diagram
\vskip 5pt

{
\hbox{\unitlength=0.5pt
\hskip 100pt
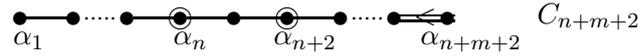
\begin{picture}(250,30)
\put(10,10){\circle*{10}}
\put(5,-10){$\alpha_1$}
 \put(15,10){\line (1,0){30}}
\put(50,10){\circle*{10}}
\put(60,10){\circle*{1}}
\put(65,10){\circle*{1}}
\put(70,10){\circle*{1}}
\put(75,10){\circle*{1}}
\put(80,10){\circle*{1}}
\put(90,10){\circle*{10}}
\put(95,10){\line (1,0){30}}
\put(130,10){\circle*{10}}
\put(130,10){\circle{16}}
\put(125,-10){$\alpha_n$}
 \put(135,10){\line (1,0){30}}
\put(170,10){\circle*{10}}
\put(175,10){\line (1,0){30}}
 \put(210,10){\circle*{10}}
\put(210,10){\circle{16}}
\put(200,-10){$\alpha_{n+2}$}
\put(215,10){\line (1,0){30}}
 \put(250,10){\circle*{10}}
  \put(260,10){\circle*{1}}
 \put(265,10){\circle*{1}}
\put(270,10){\circle*{1}}
\put(275,10){\circle*{1}}
\put(280,10){\circle*{1}}
\put(290,10){\circle*{10}}
\put(294,12){\line (1,0){41}}
\put(294,8){\line (1,0){41}}
\put(305,5.5){$<$}
\put(330,10){\circle*{10}}
\put(310,-10){$\alpha_{n+m+2}$}
\end{picture} 
\hskip 70pt\raisebox{3pt}{$C_{n+m+2}$}}

 \vskip 10pt
Up to geometric equivalence we can take here   $G= GL(n)\times GL(2)\times Sp(m)$ acting on $V=V_{1}\oplus V_{2}$  where $V_{1}=M_{n,2}$ and $V_{2}=M_{2m,2}$ by
$$(g_{1},g_{2},g_{3})(u,v)=(g_{1}u{^t\kern -1pt g_{2}},g_{3}v\kern 1pt{^t\kern -1ptg_{2}}),\,\,  g_{1}\in GL(n),\, g_{2}\in SL(2),\, g_{3}\in Sp(m).$$
\vskip 5pt

a) Let us first consider the case where $n>2$. The action of $G$ on $V_{2}$ reduces to the action $GL(2)\times Sp(m)$ on $V_{2}$ which is of parabolic type corresponding to the subdiagram
\vskip 5pt
\raisebox{8pt}{
\hbox{\unitlength=0.5pt
\hskip 100pt\begin{picture}(300,30)
\put(50,10){\circle*{10}}
\put(55,10){\line (1,0){30}}
\put(90,10){\circle*{10}}
\put(90,10){\circle{16}}
\put(95,10){\line (1,0){30}}
\put(130,10){\circle*{10}}
\put(130,10){\circle*{1}}
\put(135,10){\circle*{1}}
\put(140,10){\circle*{1}}
\put(145,10){\circle*{1}}
\put(150,10){\circle*{1}}
\put(155,10){\circle*{1}}
\put(160,10){\circle*{1}}
\put(165,10){\circle*{1}}
\put(170,10){\circle*{1}}
\put(175,10){\circle*{10}}
\put(180,10){\line (1,0){30}}
\put(215,10){\circle*{10}}
\put(219,12){\line (1,0){41}}
\put(219,8){\line(1,0){41}}
\put(235,5.5){$<$}
\put(265,10){\circle*{10}}
\end{picture} \,\,\,\raisebox{3pt} {$C_{m+2}$}
}}
\vskip 5pt

This case has already been considered in $4.2.7.$ above. And we know from the calculation we did there that the generic isotropy subgroup of $(GL(2)\times Sp(m), V_{2})$ consists of certain pairs of the form $(g_{2},g_{3})$ where $g_{2}$ takes all values in $SL(2)$. Therefore the generic isotropy subgroup of $(G,V_{2})$ acting on $V_{1}$  is the representation $(GL(n)\times SL(2),V_{1})$ with $n>2$. As this representation has no relative invariant we can apply Proposition \ref{prop-isotropie-rang1}, and obtain that this $MF$ space is $QD1$. The fundamental relative invariant is given by $f(u,v)= Pf({^tv}Jv)$, where $v\in M_{2m,2}$, and where  $J=\left[\begin{array}{cc}0&Id_{m}\\
-Id_{m}&0
\end{array}\right]$. As this  invariant does only depend on $v$, the corresponding $PV$ is not regular. It is case $(8)$ in Table 3.
\vskip 5pt

b) Consider now the case where $n=2$. Here each of the two subspaces $(G,V_{1})$ and $(G,V_{2})$ has his own relative invariant (the determinant on $V_{1}$ and the preceding invariant $f(u,v)= Pf({^tv}Jv)$ on $V_{2}$). Therefore this space is not $QD1$ if $n=2$.

 \vskip5pt
\subsubsection{} $(Sp(n)\otimes SL(2))\oplus_{SL(2)}(SL(2)\otimes Sp(m)) \times {\C^*}^2,\, (n,m\geq 2)$. 
\vskip 5pt
Up to geometric equivalence we can take $G=Sp(n)\times GL(2)\times Sp(m)\times \C^*$ acting on $V=V_{1}\oplus V_{2}$ where $V_{1}=M_{2n,2}$ and $V_{2}=M_{2m,2}$ by
$$(g_{1},g_{2},g_{3},\lambda).(X,Y)=(g_{1}X{^t \kern -1pt g_{2}},\lambda g_{3}Y{^t \kern -1pt g_{2}}),$$
where $g_{1}\in Sp(n), g_{2}\in GL(2), g_{3}\in Sp(m),\lambda \in \C^*, X\in M_{2n,2}, Y\in M_{2m,2}$.

According to the diagramatical rules in Remark \ref{rem-diagram} d) this space is not of parabolic type..

As each of the representations $(G,V_{1})$ 	and $(G,V_{2})$ is of the type seen in 4.1.7. above, they have each their own fundamental relative invariant. Therefore this $MF$ space is not $QD1$.

 \vskip5pt
\subsubsection{} $Sp(n)\oplus_{Sp(n)} Sp(n) \times {\C^*}^2,\, (n\geq 2)$. 
\vskip 5pt
 
Here $G=Sp(n)\times {\C^*}^2$ acts on $V= M_{2n,1}\oplus M_{2n,1}$ by
$$(g,\lambda,\mu).(u,v)=(\lambda gu, \mu gv),\, g\in Sp(n),\lambda,\mu\in \C,\,u,v\in M_{2n,1}.$$
At the infinitesimal level the Lie algebra ${\go g}= {\go{sp}}(n)\times \C^2$ acts on $V$ by
$$(x,\lambda, \mu)(u,v)=(\lambda u+ x u, \mu v+ x v), \, x\in {\go{sp}}(n),\lambda,\mu\in \C,\,u,v\in M_{2n,1}.$$
First of all let us remark that there is at least one fundamental relative invariant, namely $f(u,v)= {^t\kern-1pt uJv}$ where $J=\left[\begin{array}{cc}0&Id_{n}\\
-Id_{n}&0
\end{array}\right]$.

Consider the element $X_{0}=(e_{1},e_{n+1})\in M_{2n,1}\oplus M_{2n,1}$ where $e_{j}$ is the $j$-th vector of the canonical base of $M_{2n,1}\simeq \C^{2n}$. An easy calculation shows that the isotropy subalgebra ${\go g}_{X_{0}}$ of $X_{0}$ is given by:
$${\go g}_{X_{0}}=\{(\left[\begin{array}{c|c} \begin{array}{cc}-\lambda&0\\
0&A\end{array}&\begin{array}{cc}0&0\\
0&B\end{array}\\ \hline
\begin{array}{cc}0&0\\ 
0&C\end{array}&\begin{array}{cc}\lambda&0\\
0&-{^t\kern-1pt A}\end{array}
\end{array}\right],\lambda,-\lambda),\, A\in {\go{gl}}({n-1}), B,C\in Sym(n-1),\lambda\in \C^*\}.$$
As $\dim {\go g}-\dim {\go g}_{X_{0}}=\dim V$, the point $X_{0}$ is generic. As ${\go g}_{X_{0}}$ is the Lie algebra of a reductive subgroup, this $PV$ is regular. Let $\widetilde{\go g}$ the Lie algebra generated by $[{\go g},{\go g}]=  {\go sp}(n)\times \{0\}\times\{0\}$ and ${\go g}_{X_{0}}$. We have  $\dim ({\go g}/\widetilde{\go g})=1$. Then according to Remark \ref{rem-caracteres}, the polynomial $f(u,v)= {^t\kern-1pt uJv}$ is the only fundamental relative invariant. Therefore this space is $QD1$. According to Remark \ref{rem-diagram} $d)$, it is not of parabolic type. It is case $(9)$ in Table 3.

 \vskip5pt
\subsubsection{} $Spin(8)\oplus_{Spin(8)} SO(8) \times {\C^*}^2$.

Let $\rho$ be one of the Spin representations of  $Spin(8)$. Here $G=Spin(8)\times {\C^*}^2$ acts on $V=\C^8\oplus \C^8$ by
$$(g,\lambda,\mu)(u,v)= (\lambda gu, \mu \rho(g)v), g\in Spin(8), \lambda,\mu\in \C^*, u,v\in \C^8.$$
This is a parabolic $PV$ in $E_{6} $ corresponding to the diagram:

 \hskip 80pt\hbox{\unitlength=0.5pt
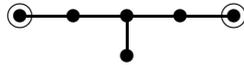
\begin{picture}(300,60)(-82,-5)
\put(10,0){\circle*{10}}
 \put(10,0){\circle{18}}
\put(15,0){\line  (1,0){30}}
\put(50,0){\circle*{10}}
 
\put(55,0){\line  (1,0){30}}
\put(90,0){\circle*{10}}
 
\put(90,-5){\line  (0,-1){20}}
\put(90,-30){\circle*{10}}
 
\put(95,0){\line  (1,0){30}}
\put(130,0){\circle*{10}}
 
\put(135,0){\line  (1,0){30}}
\put(170,0){\circle*{10}}
\put(170,0){\circle{18}}
 
  \end{picture}
  }

\vskip 20pt
As each of the two summands of this representation has his own fundamental relative invariant (a quadratic form), this space is not $QD1$.

\vfill\eject
{\scriptsize

\centerline{ \bf  Tables of indecomposable, saturated, multiplicity free representations}
 
 \centerline{\bf with one dimensional quotient }
\vskip 5pt
\centerline{{\bf Table 2: Irreducible representations}}
 \centerline{(Notations for representations as in \cite{Benson-Ratcliff-article}, see section $4.1.$)}
\centerline{\begin{tabular}{|c|c|c|c|}
\hline
&&&\\
{\rm Representation, rank}
&{\rm Weighted Dynkin diagram (if parabolic type)}
 &Regular   
&Fundamental invariant
 \\
 &&&\\
\hline
\hline
\raisebox{20pt}{\vbox{\hbox{(1)}\kern 5pt\hbox{$  SO(n)\times {\bb C}^* $ ($n\geq 3$)}
\hbox{($n\geq 3$),\text{ rank=2}}}}
&  \vbox{
  \hbox{ $n=2p+1$ \unitlength=0.5pt
\begin{picture}(250,30)(-10,0)
\put(10,10){\circle*{10}}
\put(10,10){\circle{16}}
\put(15,10){\line (1,0){30}}
\put(50,10){\circle*{10}}
\put(55,10){\line (1,0){30}}
\put(90,10){\circle*{10}}
\put(95,10){\line (1,0){30}}
\put(130,10){\circle*{10}}
\put(135,10){\circle*{1}}
\put(140,10){\circle*{1}}
\put(145,10){\circle*{1}}
\put(150,10){\circle*{1}}
\put(155,10){\circle*{1}}
\put(160,10){\circle*{1}}
\put(165,10){\circle*{1}}
\put(170,10){\circle*{10}}
\put(174,12){\line (1,0){41}}
\put(174,8){\line (1,0){41}}
\put(190,5.5){$>$}
\put(220,10){\circle*{10}}
\end{picture} $B_{p+1}$
}
{\vskip20pt}
\vbox{\hbox{ \raisebox{9pt}{$n=2p$} \,\,\,\,   \unitlength=0.5pt
\begin{picture}(240,40)(0,-15)
\put(10,10){\circle*{10}}
\put(10,10){\circle{16}}
\put(15,10){\line (1,0){30}}
\put(50,10){\circle*{10}}
\put(55,10){\line (1,0){30}}
\put(90,10){\circle*{10}}
\put(95,10){\line (1,0){30}}
\put(130,10){\circle*{10}}
\put(140,10){\circle*{1}}
\put(145,10){\circle*{1}}
\put(150,10){\circle*{1}}
\put(155,10){\circle*{1}}
\put(160,10){\circle*{1}}
\put(170,10){\circle*{10}}
\put(175,10){\line (1,0){30}}
\put(210,10){\circle*{10}}
\put(215,14){\line (1,1){20}}
\put(240,36){\circle*{10}}
\put(215,6){\line(1,-1){20}}
\put(240,-16){\circle*{10}}
\end{picture}\,\,\,\,\, \raisebox{9pt}{$D_{p+1}$}}
}
\vskip 10pt
\hbox{\hskip 40pt Commutative Parabolic (both)}}
& \raisebox{20pt}{Yes}
&\raisebox{20pt} {\vbox{ \hbox{Non degenerate}
 \hbox {quadratic form}
}}\\
&&&\\
\hline
{\vbox{\hbox{(2)}\kern 3pt\hbox{$  S^2(SL(n))\times {\bb C}^* $  }
\hbox{$(n\geq 2)$,\text{ rank= n}}}}
&\hskip 40pt{
\vbox{\hbox{\unitlength=0.5pt
\begin{picture}(250,30)
\put(10,10){\circle*{10}}
\put(15,10){\line (1,0){30}}
\put(50,10){\circle*{10}}
\put(55,10){\line (1,0){30}}
\put(90,10){\circle*{10}}
\put(95,10){\circle*{1}}
\put(100,10){\circle*{1}}
\put(105,10){\circle*{1}}
\put(110,10){\circle*{1}}
\put(115,10){\circle*{1}}
\put(120,10){\circle*{1}}
\put(125,10){\circle*{1}}
\put(130,10){\circle*{1}}
\put(135,10){\circle*{1}}
\put(140,10){\circle*{10}}
\put(145,10){\line (1,0){30}}
\put(180,10){\circle*{10}}
\put(184,12){\line (1,0){41}}
\put(184,8){\line(1,0){41}}
\put(200,5.5){$<$}
\put(230,10){\circle*{10}}
\put(230,10){\circle{16}}
\end{picture} \,\,\,\,\, $C_{n}$
}
\vskip 5pt
\hbox{Commutative Parabolic}}
}
  &Yes
& \raisebox{-2pt}{{\vbox{ \hbox{Determinant on }
 \hbox  {symmetric matrices}}}}\\
&&&\\
\hline
&&&\\
{\vbox{\hbox{(3)}\kern 3pt\hbox{$  \Lambda^2(SL(n))\times {\bb C}^*$}
\hbox{($n\geq 4$) and $n=2p$}
\hbox{\text{rank=p}}}}
&\hskip 40pt{
\vbox{\hbox{\unitlength=0.5pt
\begin{picture}(240,35)(0,-15)
\put(10,0){\circle*{10}}
\put(15,0){\line (1,0){30}}
\put(50,0){\circle*{10}}
\put(55,0){\line (1,0){30}}
\put(90,0){\circle*{10}}
\put(95,0){\line (1,0){30}}
\put(130,0){\circle*{10}}
\put(140,0){\circle*{1}}
\put(145,0){\circle*{1}}
\put(150,0){\circle*{1}}
\put(155,0){\circle*{1}}
\put(160,0){\circle*{1}}
\put(170,0){\circle*{10}}
\put(175,0){\line (1,0){30}}
\put(210,0){\circle*{10}}
\put(215,4){\line (1,1){20}}
\put(240,26){\circle*{10}}
\put(240,26){\circle{16}}
\put(215,-4){\line (1,-1){20}}
\put(240,-26){\circle*{10}}
\end{picture}\hskip 20pt\raisebox {8pt}{$D_{2p}$}
}
\vskip 5pt
\hbox{Commutative Parabolic}}
}
&\raisebox{5pt}{Yes} & \raisebox{5pt}{{\vbox{ \hbox{pfaffian on skew}
 \hbox  {symmetric matrices}}}}\\
 &&&\\
 \hline
{\vbox{ \hbox{(4)}\kern 5pt\hbox{$ (SL(n)^*\otimes SL(n))\times {\bb C}^*$ }
\hbox{ $(n\geq 2)$, \text{rank=n}}}} 
& \hskip 40pt{
\vbox{\hbox{\unitlength=0.5pt
\hskip-80pt
\begin{picture}(400,30)(0,10)
\put(90,10){\circle*{10}}
\put(85,-10){$\alpha_1$}
\put(95,10){\line (1,0){30}}
\put(130,10){\circle*{10}}
 \put(140,10){\circle*{1}}
\put(145,10){\circle*{1}}
\put(150,10){\circle*{1}}
\put(155,10){\circle*{1}}
\put(160,10){\circle*{1}}
\put(165,10){\circle*{1}}
\put(170,10){\circle*{1}}
\put(175,10){\circle*{1}}
\put(180,10){\circle*{1}}
 \put(195,10){\circle*{10}}
 \put(195,10){\line (1,0){30}}
\put(230,10){\circle*{10}}
\put(220,-10){$\alpha_{n}$}
\put(230,10){\circle{18}}
\put(235,10){\line (1,0){30}}
\put(270,10){\circle*{10}}
 \put(280,10){\circle*{1}}
\put(285,10){\circle*{1}}
\put(290,10){\circle*{1}}
\put(295,10){\circle*{1}}
\put(300,10){\circle*{1}}
\put(305,10){\circle*{1}}
\put(310,10){\circle*{1}}
\put(315,10){\circle*{1}}
\put(320,10){\circle*{1}}
\put(330,10){\circle*{10}}
\put(335,10){\line (1,0){30}}
\put(370,10){\circle*{10}}
\put(360,-10){$\alpha_{2n-1}$}
\end{picture} 
\hskip 20pt\raisebox {-4pt}{$A_{2p-1}$}
}
\vskip 15pt
\hbox{\hskip 15pt Commutative Parabolic}}
}
  & \raisebox{10pt}{Yes} &  \raisebox{10pt}{{\vbox{ \hbox{Determinant on full}
 \hbox  {matrix space }}}}\\
&&&\\
\hline
&&&\\
 \raisebox{10pt}{{\vbox{\hbox{(5)}\kern 5pt\hbox{$ E_{6}  \times {\bb C}^*$ (dim=27)}
 \hbox{\text{rank=3}}}}}
& 
\hskip 40pt{
\vbox{\hbox{\unitlength=0.5pt
\begin{picture}(240,40)(-10,-45)
\put(10,0){\circle*{10}}
\put(15,0){\line  (1,0){30}}
\put(50,0){\circle*{10}}
\put(55,0){\line  (1,0){30}}
\put(90,0){\circle*{10}}
\put(90,-5){\line  (0,-1){30}}
\put(90,-40){\circle*{10}}
\put(95,0){\line  (1,0){30}}
\put(130,0){\circle*{10}}
\put(135,0){\line  (1,0){30}}
\put(170,0){\circle*{10}}
\put(175,0){\line (1,0){30}}
\put(210,0){\circle*{10}}
\put(210,0){\circle{16}}
\end{picture}\hskip 20pt\raisebox {20pt}{$E_{7}$}
}
\vskip 5pt
\hbox{\hskip 5ptCommutative Parabolic}}
}
&\raisebox{10pt}{Yes} &  \raisebox{10pt}{Freudenthal cubic}\\
&&&\\
\hline
&&&\\
 \raisebox{0pt}{{\vbox{\hbox{(6)}\kern 5pt\hbox{$ (SL(2)\otimes Sp(n))\times {\bb C}^*$}
 \hbox{($n\geq 2$),\text{ rank=3}}}}}
 & \hskip 10pt\raisebox{5pt}{{
\hbox{\unitlength=0.5pt
\begin{picture}(250,30)
\put(10,10){\circle*{10}}
\put(15,10){\line (1,0){30}}
\put(50,10){\circle{16}}
\put(50,10){\circle*{10}}
\put(55,10){\line (1,0){30}}
\put(90,10){\circle*{10}}
\put(95,10){\circle*{1}}
\put(100,10){\circle*{1}}
\put(105,10){\circle*{1}}
\put(110,10){\circle*{1}}
\put(115,10){\circle*{1}}
\put(120,10){\circle*{1}}
\put(125,10){\circle*{1}}
\put(130,10){\circle*{1}}
\put(135,10){\circle*{1}}
\put(140,10){\circle*{10}}
\put(145,10){\line (1,0){30}}
\put(180,10){\circle*{10}}
\put(184,12){\line (1,0){41}}
\put(184,8){\line(1,0){41}}
\put(200,5.5){$<$}
\put(230,10){\circle*{10}}
\end{picture} \,\,\,\,\, $C_{n+2}$
}}}
& \raisebox{5pt}{Yes}&\vbox{\hbox {$Pf(^t{X}JX)$}
\hbox{$X\in M(2n,2)$}
\hbox{$Pf$=pfaffian\, of}
\hbox{$2\times2$\, matrices}}\\
&&&\\
\hline
&&&\\
\raisebox{0pt}{\vbox{\hbox{(7)}\kern 5pt\hbox{$  SL(4)\times Sp(2)\times {\bb C}^*$}
\hbox{\text{rank=6}}}}
& \raisebox{15pt}{{
\hbox{\unitlength=0.5pt
\begin{picture}(250,30)
\put(10,10){\circle*{10}}
\put(15,10){\line (1,0){30}}
\put(50,10){\circle*{10}}
\put(55,10){\line (1,0){30}}
\put(90,10){\circle*{10}}
\put(95,10){\line (1,0){40}}
\put(140,10){\circle*{10}}
\put(140,10){\circle{16}}
\put(145,10){\line (1,0){30}}
\put(180,10){\circle*{10}}
\put(184,12){\line (1,0){41}}
\put(184,8){\line(1,0){41}}
\put(200,5.5){$<$}
\put(230,10){\circle*{10}}
\end{picture} \,\,\,\,\, $C_{6}$
}}}&\raisebox{15pt}{ Yes} & \raisebox{15pt}{ Det(X),\, $X\in M(4)$}\\
\hline
&&&\\
\raisebox{0pt}{\vbox{\hbox{(8)}\kern 5pt\hbox{$  Spin(7)\times {\bb C}^* $}
\hbox{\text{rank=2}}}}&\hskip 30pt\raisebox{10pt}{{
\hbox{\unitlength=0.5pt
\begin{picture}(250,30)
\put(10,10){\circle*{10}}
\put(15,10){\line (1,0){30}}
\put(50,10){\circle*{10}}
\put(55,12){\line (1,0){41}}
\put(65,5.5){$>$}
\put(55,8){\line(1,0){41}}
\put(90,10){\circle*{10}}
\put(95,10){\line (1,0){40}}
\put(140,10){\circle*{10}}
\put(140,10){\circle{16}}
\end{picture}  $F_{4}$
}}} &\raisebox{5pt}{Yes}&\raisebox{5pt}{{\vbox{ \hbox{Non degenerate}
 \hbox {quadratic form}
 \hbox{($Spin(7)\hookrightarrow SO(8)$)}
}}}\\
\hline
&&&\\
\raisebox{0pt}{\vbox{\hbox{(9)}\kern 5pt\hbox{$ Spin(9) \times {\bb C}^*$}
\hbox{\text{rank=3}}}}& \raisebox{10pt}{Non parabolic} &\raisebox{5pt}{Yes}&\raisebox{5pt}{{\vbox{ \hbox{Non degenerate}
 \hbox {quadratic form}
 }}}\\
&&&\\
\hline
&&&\\
{\vbox{\hbox{(10)}\kern 5pt\hbox{$  G_{2}\times {\bb C}^*$ $(\dim =7)$}
\hbox{\text{rank=2}}}}& \raisebox{10pt}{Non parabolic} & \raisebox{10pt}{Yes}&\raisebox{5pt}{{\vbox{ \hbox{Non degenerate}
 \hbox {quadratic form}
  \hbox {$G_{2}\hookrightarrow SO(7)$}
 }}}\\
\hline
\end{tabular}}
 }

\pagebreak
{\scriptsize
 \centerline{{\bf Table 3: Non Irreducible representations}}
 
 \centerline{(Notations for representations as in \cite{Benson-Ratcliff-article}, see section $4.2.$)}
     \vskip 10pt
\centerline{\begin{tabular}{|c|c|c|c|}
\hline
&&&\\
  {\rm Representation} 
&{\rm Weighted Dynkin diagram (if parabolic type)}
 &Regular   
&Fundamental invariant\\
 &&&\\
\hline
\hline
&&&\\
{{\vbox{ \hbox{(1)}\kern 2pt
 \hbox {$ ( SL(n)^*\oplus_{_{SL(n)}}SL_{n})\times ({\bb C}^*)^2$
}
 \hbox {$ n\geq 2$
}
\hbox{\text{rank=3}}
 }}} 
&{
\hbox{\unitlength=0.5pt
\begin{picture}(250,30)
\put(10,10){\circle*{10}}
\put(10,10){\circle{16}}
\put(15,10){\line (1,0){30}}
\put(50,10){\circle*{10}}
\put(60,10){\circle*{1}}
\put(65,10){\circle*{1}}
\put(70,10){\circle*{1}}
\put(75,10){\circle*{1}}
\put(80,10){\circle*{1}}
\put(90,10){\circle*{10}}
\put(95,10){\line (1,0){30}}
\put(130,10){\circle*{10}}
 \put(135,10){\line (1,0){30}}
\put(170,10){\circle*{10}}
\put(180,10){\circle*{1}}
\put(185,10){\circle*{1}}
\put(190,10){\circle*{1}}
\put(195,10){\circle*{1}}
\put(200,10){\circle*{1}}
\put(210,10){\circle*{10}}
\put(215,10){\line (1,0){30}}
\put(250,10){\circle*{10}}
\put(250,10){\circle{16}}
\end{picture}
}}\hskip 10pt \raisebox{3pt}{$A_{n+1}$}
&\raisebox{3pt}{Yes}&{{\vbox{ \hbox{$f(u,v)=uv$}
 \hbox {on $M(1,n)\oplus M(n,1)$}
 }}}\\
&&&\\
\hline
&&&\\
\raisebox{10pt} {\vbox{\hbox{(2)(a)}\kern 2pt\hbox{$ (SL(n)\oplus_{_{SL(n)}}\Lambda^2(SL(n))\times ({\bb C}^*)^2$ \ }  
\hbox {($n\geq4, n=2p \text{  even}$)} 
\hbox{\text{rank=n=2p}}
 }}
 &\raisebox{0pt}{\vbox{\hbox{non parabolic except for the case: } \vskip 5pt {
\hbox{\unitlength=0.5pt
\begin{picture}(200,30)(5,-60)
\put(10,0){\circle*{10}}
\put(10,0){\circle{16}}
\put(15,0){\line  (1,0){30}}
\put(50,0){\circle*{10}}
\put(55,0){\line  (1,0){30}}
\put(90,0){\circle*{10}}
\put(90,-5){\line  (0,-1){30}}
\put(90,-40){\circle*{10}}
\put(90,-40){\circle{16}}
\put(95,0){\line  (1,0){30}}
\put(130,0){\circle*{10}}
\put(135,0){\line  (1,0){30}}
\put(170,0){\circle*{10}}
\put(175,0){\line (1,0){30}}
\put(210,0){\circle*{10}}
\end{picture} \hskip 20pt\raisebox {27pt}{$E_{7}, n=6$}
}} }}&\raisebox{20pt}{No}& \raisebox{15pt}{{\vbox{ \hbox{pfaffian on skew}
 \hbox  {symmetric matrices}\hbox{(on 2nd component)}} }}\\
&&& \\
\hline
&&&\\
\raisebox{30pt} {\vbox{\hbox{(2)(b)}\kern 2pt\hbox{$ (SL(n)\oplus_{_{SL(n)}}\Lambda^2(SL(n))\times ({\bb C}^*)^2$ }  
\hbox {($n\geq4, n=2p+1 \text{  odd}$)} 
\hbox{\text{rank=n=2p+1}}
 }}
 &\raisebox{10pt}{\vbox{\hbox{non parabolic except for the cases: } \vskip 5pt {
\hbox{\unitlength=0.5pt
\begin{picture}(190,40)(-10,-45)
\put(10,0){\circle*{10}}
\put(10,0){\circle{16}}
\put(15,0){\line  (1,0){30}}
\put(50,0){\circle*{10}}
\put(55,0){\line  (1,0){30}}
\put(90,0){\circle*{10}}
\put(90,-5){\line  (0,-1){30}}
\put(90,-40){\circle*{10}}
\put(90,-40){\circle{16}}
\put(95,0){\line  (1,0){30}}
\put(130,0){\circle*{10}}
\put(135,0){\line  (1,0){30}}
\put(170,0){\circle*{10}}
\end{picture} \hskip 0pt\raisebox {20pt}{$E_{6}, n=5$,}
  }}
 \vskip 7pt {\hbox{\unitlength=0.5pt
\begin{picture}(270,40)(-10,-45)
\put(10,0){\circle*{10}}
\put(10,0){\circle{16}}
\put(15,0){\line  (1,0){30}}
\put(50,0){\circle*{10}}
\put(55,0){\line  (1,0){30}}
\put(90,0){\circle*{10}}
\put(90,-5){\line  (0,-1){30}}
\put(90,-40){\circle*{10}}
\put(90,-40){\circle{16}}
\put(95,0){\line  (1,0){30}}
\put(130,0){\circle*{10}}
\put(135,0){\line  (1,0){30}}
\put(170,0){\circle*{10}}
\put(175,0){\line (1,0){30}}
\put(210,0){\circle*{10}}
\put(215,0){\line (1,0){30}}
\put(250,0){\circle*{10}}
\end{picture} \hskip 0pt\raisebox {20pt}{$E_{8}, n=7$.}
}} }
  }&\raisebox{50pt}{Yes}& \raisebox{40pt}{{\vbox{ \hbox{$f(v,x)=$}\hbox{ $Pf(\left[
\begin{array}{c|c}
x &v\\ \hline
-^tv &0
\end{array}\right])$}
\vskip 5pt
 \hbox  { $v\in {\C^n}$}\hbox{$x\in AS(2p+1)$}} }}\\
 \hline
&&&\\
\raisebox{0pt} {\vbox{\hbox{(3)}\kern 2pt\hbox{$ (SL(n)^*\oplus_{_{SL(n)}}\Lambda^2(SL(n))\times ({\bb C}^*)^2$ \ }  
\hbox {($n\geq4, n=2p \text{ even}$)} 
\hbox{\text{rank=n}}
 }}
&\raisebox{10pt}{\hbox{\unitlength=0.5pt
\begin{picture}(240,35)(0,-15)
\put(10,0){\circle*{10}}
\put(10,0){\circle{16}}
\put(15,0){\line (1,0){30}}
\put(50,0){\circle*{10}}
\put(55,0){\line (1,0){30}}
\put(90,0){\circle*{10}}
\put(95,0){\line (1,0){30}}
\put(130,0){\circle*{10}}
\put(140,0){\circle*{1}}
\put(145,0){\circle*{1}}
\put(150,0){\circle*{1}}
\put(155,0){\circle*{1}}
\put(160,0){\circle*{1}}
\put(170,0){\circle*{10}}
\put(175,0){\line (1,0){30}}
\put(210,0){\circle*{10}}
\put(215,4){\line (1,1){20}}
\put(240,26){\circle*{10}}
\put(240,26){\circle{16}}
\put(215,-4){\line (1,-1){20}}
\put(240,-26){\circle*{10}}
\end{picture}\hskip 20pt\raisebox {5pt}{$D_{2p}$}
}}
 &\raisebox{10pt}{No}&\raisebox{10pt}{\vbox{ 
\hbox{$Pf(x)$\, \,\,(pfaffian)}\vskip 5pt\hbox{(on 2nd component)}}}\\
&&&\\
\hline&&&\\
\raisebox{25pt}{\vbox{\hbox{(4)(a)}\kern 2pt\hbox{$  SL(n)\oplus_{_{SL(n)}}(SL(n)\otimes SL(n))\times ({\bb C}^*)^2, n\geq 2$}} }&\raisebox{10pt}{\vbox{\hbox{non parabolic except for the cases: } \vskip 5pt {\hbox{\unitlength=0.5pt
\begin{picture}(230,30)(70,0)
\put(90,0){\circle*{10}}
\put(95,0){\line (1,0){30}}
\put(130,0){\circle*{10}}
\put(135,0){\line (1,0){30}}
\put(170,0){\circle*{10}}
\put(170,0){\circle{16}}
\put(175,0){\line (1,0){30}}
\put(210,0){\circle*{10}}
\put(215,4){\line (1,1){20}}
\put(240,26){\circle*{10}}
\put(240,26){\circle{16}}
\put(215,-4){\line (1,-1){20}}
\put(240,-26){\circle*{10}}
\end{picture}\hskip -5pt\raisebox {0pt}{$D_{6},\, n=3$}
}}
 \vskip 7pt {\hbox{\unitlength=0.5pt
\begin{picture}(270,40)(-10,-10)
\put(10,0){\circle*{10}}
\put(15,0){\line  (1,0){30}}
\put(50,0){\circle*{10}}
\put(55,0){\line  (1,0){30}}
\put(90,0){\circle*{10}}
\put(90,-5){\line  (0,-1){30}}
\put(90,-40){\circle*{10}}
\put(90,-40){\circle{16}}
\put(95,0){\line  (1,0){30}}
\put(130,0){\circle*{10}}
\put(130,0){\circle{16}}
\put(135,0){\line  (1,0){30}}
\put(170,0){\circle*{10}}
\put(175,0){\line (1,0){30}}
\put(210,0){\circle*{10}}
\put(215,0){\line (1,0){30}}
\put(250,0){\circle*{10}}
\end{picture} \hskip 0pt\raisebox {5pt}{$E_{8}, n=4$.}
}} }
  }&\raisebox{30pt}{No}&\raisebox{30pt}{\vbox{ 
\hbox{Determinant}\vskip 5pt\hbox{(on 2nd component)}}}\\
&&&\\
\hline&&&\\
\raisebox{25pt}{\vbox{\hbox{(4)(b)}\kern 2pt\hbox{$(  SL(n)\oplus_{_{SL(n)}}(SL(n)\otimes SL(n-1))\times ({\bb C}^*)^2, n\geq 3$}} }&\raisebox{10pt}{\vbox{\hbox{non parabolic except for the cases: } \vskip 5pt {\hbox{\unitlength=0.5pt
\begin{picture}(230,30)(110,0)
\put(130,0){\circle*{10}}
\put(135,0){\line (1,0){30}}
\put(170,0){\circle*{10}}
\put(170,0){\circle{16}}
\put(175,0){\line (1,0){30}}
\put(210,0){\circle*{10}}
\put(215,4){\line (1,1){20}}
\put(240,26){\circle*{10}}
\put(240,26){\circle{16}}
\put(215,-4){\line (1,-1){20}}
\put(240,-26){\circle*{10}}
\end{picture}\hskip -25pt\raisebox {0pt}{$D_{5},\, n=3$}
}}
 \vskip 7pt {\hbox{\unitlength=0.5pt
\begin{picture}(270,40)(-10,-10)
\put(10,0){\circle*{10}}
\put(15,0){\line  (1,0){30}}
\put(50,0){\circle*{10}}
\put(55,0){\line  (1,0){30}}
\put(90,0){\circle*{10}}
\put(90,-5){\line  (0,-1){30}}
\put(90,-40){\circle*{10}}
\put(90,-40){\circle{16}}
\put(95,0){\line  (1,0){30}}
\put(130,0){\circle*{10}}
\put(130,0){\circle{16}}
\put(135,0){\line  (1,0){30}}
\put(170,0){\circle*{10}}
\put(175,0){\line (1,0){30}}
\put(210,0){\circle*{10}}
\end{picture} \hskip -10pt\raisebox {5pt}{$E_{7}, n=4$.}
}} }
  }&\raisebox{30pt}{Yes}&\raisebox{30pt}{\vbox{ 
\hbox{$\det(v;x)$}\vskip 5pt\hbox{$v\in M_{n,1},x\in M_{n,n-1}$}}}\\
&&&\\
\hline
&&&\\
\vbox{\hbox{(5)}\kern 2pt \hbox{$ SL(n)^*\oplus_{_{SL(n)}}(SL(n)\otimes SL(n))\times({\bb C}^*)^2, n\geq 3$}
 \hbox{\text{rank=2n}}}
&\hskip 70pt\raisebox{8pt} {\hbox{\unitlength=0.5pt
\hskip-100pt \begin{picture}(400,30)(0,5)
\put(90,10){\circle*{10}}
\put(90,10){\circle{16}}
\put(85,-10){$\alpha_1$}
\put(95,10){\line (1,0){30}}
\put(130,10){\circle*{10}}
 \put(140,10){\circle*{1}}
\put(145,10){\circle*{1}}
\put(150,10){\circle*{1}}
\put(155,10){\circle*{1}}
\put(160,10){\circle*{1}}
\put(165,10){\circle*{1}}
\put(170,10){\circle*{1}}
\put(175,10){\circle*{1}}
\put(180,10){\circle*{1}}
 \put(195,10){\circle*{10}}
 \put(195,10){\line (1,0){30}}
\put(230,10){\circle*{10}}
\put(220,-10){$\alpha_{n+1}$}
\put(230,10){\circle{18}}
\put(235,10){\line (1,0){30}}
\put(270,10){\circle*{10}}
 \put(280,10){\circle*{1}}
\put(285,10){\circle*{1}}
\put(290,10){\circle*{1}}
\put(295,10){\circle*{1}}
\put(300,10){\circle*{1}}
\put(305,10){\circle*{1}}
\put(310,10){\circle*{1}}
\put(315,10){\circle*{1}}
\put(320,10){\circle*{1}}
\put(330,10){\circle*{10}}
\put(335,10){\line (1,0){30}}
\put(370,10){\circle*{10}}
\put(360,-10){$\alpha_{2n}$}
\end{picture} 
}}  \raisebox{5pt}{$A_{2n}$}
& \raisebox{5pt}{No}& \raisebox{3pt}{{\vbox{ \hbox{Determinant   }
\vskip 5pt \hbox{(on 2nd component)}} }}\\
 &&&\\
 \hline
 \end{tabular}}}
 \vskip 5pt
\hfill  Continued next page.
 \pagebreak

{\scriptsize
 \centerline{{\bf Table 3(continued): Non Irreducible representations}}
  \centerline{(Notations for representations as in \cite{Benson-Ratcliff-article}, see section 4.2)}
     \vskip 10pt
\centerline{\begin{tabular}{|c|c|c|c|}
\hline
&&&\\
{\rm Representation, rank}
&{\rm Weighted Dynkin diagram (if parabolic type)}
 &Regular   
&Fundamental invariant
 \\
 &&&\\
\hline
\hline
&&&\\
{{\vbox{ \hbox{(6)}\kern 2pt
 \hbox {$ ( SL(2)\oplus_{_{SL(2)}}(SL(2)\otimes Sp(n))$}
 \hbox{($\times{\bb C}^*)^2$}
 \hbox {$ n\geq 2$}
\hbox{\text{rank=3}}
 }}} 
 & \hskip 70pt\raisebox{15pt}{ \hbox{\unitlength=0.5pt
\hskip -50pt\begin{picture}(280,20)
\put(10,10){\circle*{10}}
\put(10,10){\circle{16}}
\put(15,10){\line (1,0){30}}
\put(50,10){\circle*{10}}
\put(55,10){\line (1,0){30}}
\put(90,10){\circle*{10}}
\put(90,10){\circle{16}}
\put(95,10){\line (1,0){30}}
\put(130,10){\circle*{10}}
\put(130,10){\circle*{1}}
\put(135,10){\circle*{1}}
\put(140,10){\circle*{1}}
\put(145,10){\circle*{1}}
\put(150,10){\circle*{1}}
\put(155,10){\circle*{1}}
\put(160,10){\circle*{1}}
\put(165,10){\circle*{1}}
\put(170,10){\circle*{1}}
\put(175,10){\circle*{10}}
\put(180,10){\line (1,0){30}}
\put(215,10){\circle*{10}}
\put(219,12){\line (1,0){41}}
\put(219,8){\line(1,0){41}}
\put(235,5.5){$<$}
\put(265,10){\circle*{10}}
\end{picture} \,\,\,\raisebox{3pt} {$C_{n+3}$}
}}
&\raisebox{5pt}{No}&\vbox{\hbox {$Pf(^t{X}JX)$}
\hbox{$X\in M(2n,2)$}
\hbox{$Pf=pfaffian$}\hbox{(on 2nd component)}}\\
&&&\\
\hline
 &&&\\
\vbox{\hbox{(7)}\kern 2pt\hbox{$(SL(2)\otimes SL(2))\oplus_{_{SL(2)}}(SL(2)\otimes SL(n))$}\hbox{$\times (\C^*)^2$}\hbox{ ($n\geq 3$)} 
\hbox{\text{rank=5}}}&
\hskip 70pt\raisebox{25pt} {\hbox{\unitlength=0.5pt
\hskip-100pt \begin{picture}(400,30)(0,10)
\put(90,10){\circle*{10}}
\put(85,-10){$\alpha_1$}
\put(95,10){\line (1,0){30}}
\put(130,10){\circle*{10}}
\put(120,-10){$\alpha_{2}$}
\put(130,10){\circle{16}}
 \put(135,10){\line (1,0){30}}
 \put(170,10){\circle*{10}}
 \put(170,10){\line (1,0){30}}
\put(205,10){\circle*{10}}
\put(195,-10){$\alpha_{4}$}
\put(205,10){\circle{16}}
\put(210,10){\line (1,0){30}}
\put(245,10){\circle*{10}}
 \put(255,10){\circle*{1}}
\put(260,10){\circle*{1}}
\put(265,10){\circle*{1}}
\put(270,10){\circle*{1}}
\put(275,10){\circle*{1}}
\put(280,10){\circle*{1}}
\put(285,10){\circle*{1}}
\put(290,10){\circle*{1}}
\put(295,10){\circle*{1}}
\put(305,10){\circle*{10}}
\put(310,10){\line (1,0){30}}
\put(345,10){\circle*{10}}
\put(335,-10){$\alpha_{n+3}$}
\end{picture} 
}}  \raisebox{20pt}{$A_{n+3}$}
&\raisebox{20pt}{No}&\raisebox{10pt}{\vbox{\hbox {$Det(X)$, $X\in M(2,2)$}
 \hbox{(on 1st component)}}}\\
&&&\\
\hline
&&&\\
\vbox{\hbox{(8)}\kern 2pt \hbox{$ (SL(n)\otimes SL(2))\oplus_{_{SL(2)}}(SL(2)\otimes Sp(m))$}\hbox{{$\times({\bb C}^*)^2$}}\hbox{($n\geq3, m\geq2)$}
\hbox{\text{rank=6}}}& \hskip 55pt\raisebox{25pt} {\hbox{\unitlength=0.5pt
\hskip-100pt \begin{picture}(400,30)(0,10)
\put(90,10){\circle*{10}}
\put(85,-10){$\alpha_1$}
\put(95,10){\line (1,0){30}}
\put(130,10){\circle*{10}}
\put(120,-10){$\alpha_{2}$}
 \put(140,10){\circle*{1}}
\put(145,10){\circle*{1}}
\put(150,10){\circle*{1}}
\put(155,10){\circle*{1}}
\put(160,10){\circle*{1}}
\put(165,10){\circle*{1}}
\put(170,10){\circle*{1}}
\put(175,10){\circle*{1}}
\put(180,10){\circle*{1}}
 \put(195,10){\circle*{10}}
 \put(185,-10){$\alpha_{n}$}
 \put(195,10){\circle{16}}
 \put(195,10){\line (1,0){30}}
\put(230,10){\circle*{10}}
\put(235,10){\line (1,0){30}}
\put(270,10){\circle*{10}}
\put(260,-10){$\alpha_{n+2}$}
\put(270,10){\circle{16}}
 \put(280,10){\circle*{1}}
\put(285,10){\circle*{1}}
\put(290,10){\circle*{1}}
\put(295,10){\circle*{1}}
\put(300,10){\circle*{1}}
\put(305,10){\circle*{1}}
\put(310,10){\circle*{1}}
\put(315,10){\circle*{1}}
\put(320,10){\circle*{1}}
\put(330,10){\circle*{10}}
\put(335,12){\line (1,0){41}}
\put(335,8){\line(1,0){41}}
\put(347,6){$<$}
\put(372,10){\circle*{10}}
\put(355,-10){$\alpha_{n+m+2}$}
\end{picture} 
}}  \raisebox{25pt}{$C_{n+m+2}$}
&\raisebox{25pt}{No}&\vbox{\hbox {$Pf(^t{X}JX)$}
\hbox{$X\in M(2m,2)$}
\hbox{$Pf=pfaffian$}\hbox{(on 2nd component)}}\\
&&&\\
\hline
&&&\\ 
{\vbox{\hbox{(9)}\kern 2pt\hbox{$  (Sp(n)\oplus_{_{Sp(n)}}Sp(n))\times({\bb C}^*)^2, n\geq2$}
\hbox{\text{rank=4}}}}& {Non parabolic}&Yes&{{\vbox{ \hbox{$f(u,v)={^t}uJv$}
 \hbox {on $M(1,2n)\oplus M(1,2n)$}
 }}}\\
&  &&\\
\hline
\end{tabular}}}

\vfill\eject

  \vfill\eject

\end{document}